\documentclass[10pt,reqno]{amsart}

\newcommand{\ignore}[1]{}

\newcommand{\st}[1]{}

\usepackage{xcolor}
\colorlet{red}{black}

 \normalsize

\usepackage{amsmath,amsthm,amssymb}
\usepackage{lscape}

\newcommand{\be}{\begin{equation}}
\newcommand{\ee}{\end{equation}}

\theoremstyle{plain}
\newtheorem{proposition}{Proposition}
\newtheorem{lemma}[proposition]{Lemma}

\newtheorem{theorem}[proposition]{Theorem}
\newtheorem{corollary}[proposition]{Corollary}
\theoremstyle{definition}
\newtheorem{definition}{Definition}
\theoremstyle{remark}
\newtheorem{remark}{Remark}
\newtheorem{observation}{Observation}
\newtheorem{example}{Example}

\newcommand{\EQn}{\begin{equation}}
\newcommand{\ENn}{\end{equation}}
\newcommand{\EQAn}{\begin{eqnarray}}
\newcommand{\ENAn}{\end{eqnarray}}
\newcommand{\EQA}{\begin{eqnarray*}}
\newcommand{\ENA}{\end{eqnarray*}}

\usepackage{algorithm, algorithmic}
\usepackage{graphicx}%,epsfig}
\usepackage{caption}
\usepackage{subcaption}
\usepackage{placeins}
\usepackage{float}
\theoremstyle{definition}
\usepackage{pdflscape}
\usepackage{array}
\newcolumntype{L}[1]{>{\raggedright\let\newline\\\arraybackslash\hspace{0pt}}m{#1}}
\newcolumntype{C}[1]{>{\centering\let\newline\\\arraybackslash\hspace{0pt}}m{#1}}
\newcolumntype{R}[1]{>{\raggedleft\let\newline\\\arraybackslash\hspace{0pt}}m{#1}}

\usepackage{multirow}
\usepackage{setspace}
\usepackage{booktabs}

\usepackage[comma]{natbib}
\usepackage{enumerate}
\usepackage{rotating}

\usepackage{dcolumn}
\newcommand{\commentout}[1]{}

\newcounter{commentcounter}
\long\def\symbolfootnote[#1]#2{\begingroup%
\def\thefootnote{\fnsymbol{footnote}}\footnote[#1]{#2}\endgroup}
\newcommand{\ccomment}[1]{{\footnotesize\textbf{\textcolor{red}{(C.\arabic{commentcounter})}}\symbolfootnote[4]{\texttt{\textcolor{red}
        {(C.\arabic{commentcounter})~#1}}}}\addtocounter{commentcounter}{1}}

\renewcommand{\ccomment}[1]{}

\usepackage{bcol}
\def\qed{$\square$}
\allowdisplaybreaks
\title{Path Cover and Path Pack Inequalities for the Capacitated Fixed-Charge Network Flow Problem}

\author{Alper Atamt\"urk, Birce Tezel and Simge K\"u\c{c}\"ukyavuz}
\thanks{ \noindent \hskip -1mm
A. Atamt\"urk: Department of Industrial Engineering \& Operations Research, University of California, Berkeley, CA 94720.
\texttt{atamturk@berkeley.edu}   \\
B. Tezel: Department of Industrial Engineering \& Operations Research, University of California, Berkeley, CA 94720. \texttt{btezel@berkeley.edu} \\
S. K\"u\c{c}\"ukyavuz: Department of Industrial and Systems Engineering, University of Washington, Seattle, WA 98195. \texttt{simge@uw.edu} .
}

\ignore{
\author{Alper Atamt\"urk, Birce Tezel \\ \small{Department of Industrial Engineering \& Operations Research} \\ \small{University of California, Berkeley, CA 94720} \\ \small{\texttt{atamturk@berkeley.edu}, \texttt{btezel@berkeley.edu}}  \\\\ Simge K\"u\c{c}\"ukyavuz \\ \small{Industrial and Systems Engineering} \\ \small{University of Washington, Seattle, WA 98195} \\ \small{\texttt{simge@uw.edu}} }
}

\date{}
\begin{document}
\maketitle
\begin{abstract}
Capacitated fixed-charge network flows are used to model a variety of problems in telecommunication, facility location, production planning and supply chain management.
In this paper, we investigate capacitated path substructures and derive strong and easy-to-compute \emph{path cover and path pack inequalities}. These inequalities are based on an explicit characterization of the submodular inequalities through a fast computation of parametric minimum cuts on a path, and they generalize the well-known flow cover and flow pack inequalities for the single-node relaxations of fixed-charge flow models. We provide necessary and sufficient facet conditions. Computational results demonstrate the effectiveness of the inequalities when used as cuts in a branch-and-cut algorithm.
\end{abstract}
\begin{center} July 2015; October 2016; May 2017 \end{center}

\BCOLReport{15.03}%{SIAM Journal on Optimization}

\section{Introduction}
Given a \textcolor{red}{\st{directed graph}} \textcolor{red}{directed multigraph} with demand or supply on the nodes, and capacity, fixed and variable
cost of flow on the arcs, the capacitated fixed-charge network flow (CFNF)
problem is to choose a subset of the arcs and route
the flow on the chosen arcs while satisfying the supply, demand and capacity constraints, so that the sum of fixed and variable costs is minimized.

There are numerous polyhedral studies \textcolor{red}{\st{on}} \textcolor{red}{of} the fixed-charge network flow problem. In a seminal paper \cite{W89} introduces the so-called submodular inequalities, which
subsume almost all valid inequalities known for capacitated fixed-charge networks. Although the submodular inequalities are very general, their coefficients are defined implicitly through value functions. In this paper, we give explicit valid inequalities that simultaneously  make use of the path substructures of the network as well as the arc capacities.

\ignore{
However, few give explicit valid inequalities that simultaneously  make use of the path substructures of the network as well as the arc capacities, which is the goal of the current paper. The inequalities introduced here are based on the seminal paper by \cite{}
}

For the {\it uncapacitated} fixed-charge network flow problem, \cite{vRW85} give flow path inequalities that are based on path substructures. \cite{rardin} introduce a new family of dicut inequalities and show that they describe the projection of \textcolor{red}{an} extended multicommodity formulation onto the original variables of fixed-charge network flow problem.
\cite{OW03} present a computational study on the performance of path and cut-set (dicut) inequalities.

For the {\it capacitated} fixed-charge network flow problem, almost all known valid inequalities are based on single-node relaxations.
\cite{PvRW85}, \cite{vRW86} and \cite{GNS99} give flow cover, generalized flow cover and lifted flow cover inequalities.
\cite{S:comp-class} introduces the complement class of generalized flow cover inequalities and \cite{A:fp} describes lifted flow pack inequalities. Both uncapacitated path inequalities and capacitated flow cover inequalities are highly valuable in solving a host of practical problems and are part of the suite of cutting planes implemented in modern mixed-integer programming solvers.

The \textit{path} structure arises naturally in network models of the lot-sizing problem. \cite{AM04} introduce valid inequalities for the capacitated lot-sizing problems with infinite inventory capacities. \cite{AK03} give valid inequalities for the lot-sizing problems with finite inventory and infinite production capacities. \textcolor{red}{\cite{Van} introduces path-modular inequalities for the uncapacitated fixed charge transportation problems. These inequalities are derived from a value function that is neither globally submodular nor supermodular but that exhibits sub or supermodularity under certain set selections.} \cite{vVO04} and \cite{GK11} give valid inequalities and extended formulations for uncapacitated lot-sizing with fixed charges on stocks. For uncapacitated lot-sizing with backlogging, \cite{pochet1988lot} and \cite{pochet1994polyhedra} provide valid inequalities and  \cite{kuccukyavuz2009} provide an explicit description of the convex hull.

\subsubsection*{Contributions}
In this paper we consider a generic path relaxation, with supply and/or demand nodes and capacities on incoming and outgoing arcs. By exploiting the path substructure of the network and introducing notions of \textit{path cover} and \textit{path pack} we provide two explicitly-described subclasses of the submodular inequalities of \cite{W89}. The most important consequence of the explicit derivation is that the coefficients of the submodular inequalities on a path can be computed efficiently. \textcolor{red}{\st{In particular, we show that \textit{all} coefficients of such an inequality can be computed by solving  max-flow/min-cut problems parametrically over the path in linear time.}} \textcolor{red}{In particular, we show that the coefficients of such an inequality can be computed by solving max-flow/min-cut problems parametrically over the path. Moreover, we show that \textit{all} of these coefficients can be computed with a single linear-time algorithm.} For a path with a single node, the inequalities reduce to the well-known flow cover and flow pack inequalities. Moreover, we show that the path cover and path pack inequalities dominate flow cover and flow pack inequalities for the corresponding single node relaxation of a path obtained by merging the path into a single node. We give necessary and sufficient facet-defining conditions. Finally, we report on computational experiments demonstrating the effectiveness of the proposed inequalities when used as cuts in a branch-and-cut algorithm.

\subsubsection*{Outline}
The remainder of this paper is organized as follows: In Section \ref{sec:model}, we describe the capacitated fixed-charge flow problem on a path, its formulation and the assumptions we make. In Section \ref{sec:spi}, we review the submodular inequalities, discuss their computation on a path, and introduce two explicit subclasses: path cover inequalities and path pack inequalities. In Section \ref{sec:facet}, we analyze \textcolor{red}{\st{the}} sufficient and necessary facet-defining conditions. In Section \ref{sec:comp}, we present computational experiments showing the effectiveness of the path cover and path pack inequalities compared to other network inequalities.

\section{Capacitated fixed-charge network flow on a path} \label{sec:model}

Let $G= (N^{\prime},A)$ be a directed \textcolor{red}{multigraph} with nodes $N^{\prime}$ and arcs $A$. Let $s_{N}$ and $t_{N}$ be the source and the sink nodes of $G$. Let $N:= N^{\prime}\setminus \{s_{N},t_{N}\}$. Without loss of generality, we label $N:=\{1,\dots, n\}$ such that a directed \textit{forward path} arc exists from node $i$ to node $i+1$ and a directed \textit{backward path} arc exists from node $i+1$ to node $i$ for each node $i=1,\dots,n-1$ (see Figure \ref{fig:lsntwk} for an illustration).
In Remarks \ref{rem:gen2} and \ref{rem:gen1}, we discuss how to obtain a ``path" graph $G$ from a more general \textcolor{red}{directed multigraph}. %Define $a:= |A|$ and $n:= |N|$.

\textcolor{red}{Let $E^+=\{(i,j)\in A: i = s_N, ~j \in N \}$ and $E^- = \{(i,j)\in A: i \in N, ~j = t_N \}$.} \textcolor{red}{\st{Let $E^+=\{(i,j)\in A: i \notin N, j \in N \}$ and $E^- = \{(i,j)\in A: i \in N, j \notin N \}$.}} Moreover, let us partition the sets $E^+$ and $E^-$ such that $E_k^+ = \{(i,j)\in A: i\notin N, j=k \}$ and $E_k^- = \{(i,j)\in A: i=k, j\notin N\}$ for $k\in N$. We refer to the arcs in $E^+$ and $E^-$ as \textit{non-path arcs}. Finally, let $E := E^+ \cup E^-$ be the set of all non-path arcs. For convenience, we generalize this set notation scheme. Given an arbitrary subset of non-path arcs $Y\subseteq E$, let $Y_j^+ = Y \cap E_j^+$ and $Y_j^- = Y \cap E_j^-$. 

\begin{remark}\label{rem:gen2}
Given a \textcolor{red}{directed multigraph} $\tilde{G}= (\tilde{N}, \tilde{A})$ with nodes $\tilde{N}$, arcs $\tilde{A}$ and a path that passes through nodes $N$, we can construct $G$ as described above by letting $E^+ = \{(i,j)\in \tilde{A}: i\in \tilde{N}\setminus N, j\in N\}$ and $E^- = \{(i,j)\in \tilde{A}: i\in N, j\in \tilde{N}\setminus N\}$ and letting all the arcs in $E^+$ be the outgoing arcs from a dummy source $s_N$ and all the arcs in $E^-$ to be incoming to a dummy sink $t_N$.
\end{remark}

\begin{remark}\label{rem:gen1} If there is an arc $t=(i,j)$ from node $i\in N$ to $j\in N$, where $|i-j|>1$, then we construct a relaxation by removing arc $t$, and replacing it with two arcs $t^-\in E_i^-$ and $t^+\in E_j^+$.  \textcolor{red}{If there are multiple arcs from node $i$ to node $j$, one can repeat the same procedure.} \end{remark}

\ignore{Let $[k,j]:=\bigcup_{i=k}^{j-1}\left[(i,i+1)\cup(i+1,i)\right]$ denote the union of the directed path from node $k$ to node $j$ and the directed path from node $j$ to node $k$.}
Throughout the paper, we use the following notation: Let $[k,j]=\{k,k+1, \dots,j\}$ \textcolor{red}{if $k\leq j$ and $\emptyset$ otherwise}, $c(S) = \sum_{t\in S} c_t$, $y(S) = \sum_{t\in S} y_t$, $(a)^+ = \max\{ 0, a\}$ and $d_{kj} = \sum_{t=k}^j d_t$ if $j\geq k$ and $0$ otherwise. Moreover, let $\text{dim}(A)$ denote the dimension of a polyhedron $A$ and $\text{conv}(S)$ be the convex hull of a set $S$.

\begin{figure} [h!]
\centering
	\includegraphics[scale=1]{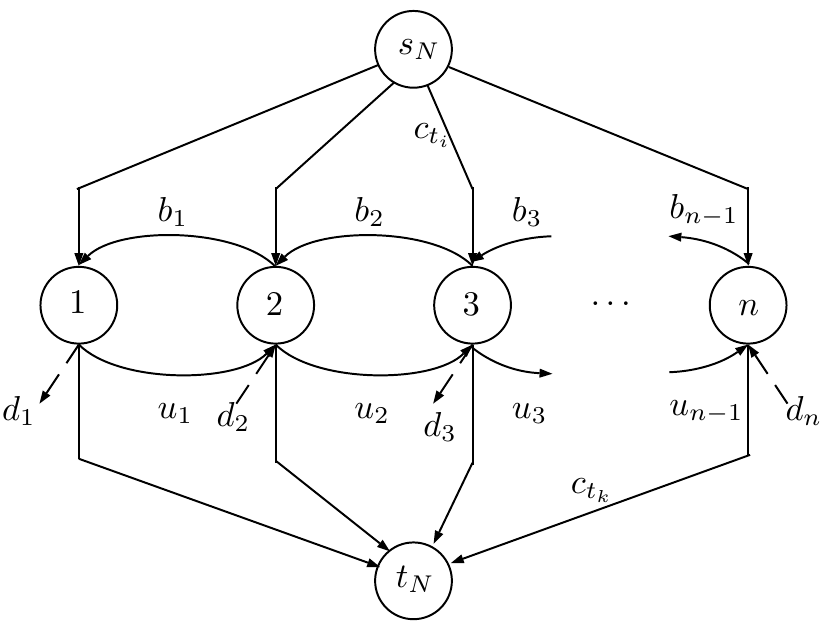}
	\caption{Fixed-charge network representation of a path.} \label{fig:lsntwk}
\end{figure}

The capacitated fixed-charge network flow problem on a path can be formulated as a mixed-integer optimization problem. \textcolor{red}{Let $d_j$ be the demand at node $j\in N$. We call a node $j\in N$ a demand node if $d_j\geq 0$ and a supply node if $d_j<0$.} Let the flow on forward path arc $(j,j+1)$ be represented by $i_j$ with an upper bound $u_j$ for $j\in N\setminus \{n\}$. Similarly, let the flow on backward path arc $(j+1,j)$ be represented by $r_j$ with an upper bound $b_j$ for $j\in N\setminus \{n\}$. Let $y_t$ be the amount of flow on arc $t\in E$ with an upper bound $c_t$. Define binary variable $x_t$ to be $1$ if $y_t>0$, and zero otherwise for all $t\in E$. An arc $t$ is \emph{closed} if $x_t=0$ and \emph{open} if $x_t=1$. Moreover, let $f_t$ be the fixed cost and $p_t$ be the unit flow cost of arc $t$. Similarly, let $h_j$ and $g_j$ be the costs of unit flow, on forward and backward arcs $(j, j+1)$ and $(j+1,j)$ respectively for $j\in N\setminus \{n\}$. Then, the problem is formulated as
\begin{subequations} \label{opt:ls}
\begin{align}
\min \quad & \sum_{t\in E} \left(f_t x_t + p_t y_t \right) + \sum_{j \in N} \left( h_j i_j + g_j r_j \right)\label{ls:first}\\
\text{s. t.} \quad & i_{j-1} - r_{j-1} + y(E_j^+) - y(E_j^-) - i_j + r_j = d_j, \quad j \in N, \label{const:flowbalance}\\
 &0\leq y_t \leq c_t x_t, \quad t\in E, \\
 &0\leq i_j \leq u_j, \quad j\in N, \\
(\text{F1})\qquad &0 \leq r_j \leq b_j, \quad j\in N, \\
& x_t \in \{0,1\}, \quad t\in E, \\
& i_0 = i_n = r_0 = r_n = 0. \label{ls:last}
\end{align}
\end{subequations}
Let $\mathcal{P}$ be the set of feasible solutions of (F\ref{opt:ls}). Figure \ref{fig:lsntwk} shows an example network representation of (F\ref{opt:ls}). \ignore{In this representation, the dummy source node $s_N$ has a supply of $d_{1n}$ and the dummy sink node $t_N$ has demand $d_{1n}$.}

Throughout we make the following assumptions on (F\ref{opt:ls}):
\begin{enumerate}[]
	\item[]\textbf{(A.1)} The set $\mathcal{P}_t=\{(x,y,i,r)\in \mathcal{P}: x_t = 0 \}\neq \emptyset$ for all $t\in E$,
	\item[]\textbf{(A.2)} $c_t > 0$, $u_j > 0$ and $b_j > 0$ for all $t\in E$ and $j\in N$,
	\item[]\textbf{(A.3)} $c_t \leq d_{1n}+c(E^-)$ for all $t\in E^+$,
	\item[]\textbf{(A.4)} $c_t \leq b_{j-1} + u_{j} + (d_j)^+ + c(E_j^-),$ for all $j \in N, t\in E_j^+ $,
	\item[]\textbf{(A.5)} $c_t \leq b_{j} + u_{j-1} + (-d_j)^+ + c(E_j^+)$ for all $j \in N, t\in E_j^-$.
\end{enumerate}
Assumptions (A.1)--(A.2) ensure that $\text{dim}\big(\text{conv}(\mathcal{P})\big) = 2|E| + |N|-2$. If (A.1) does not hold for some $t\in E$, then $x_t=1$ for all points in $\mathcal{P}$. Similarly, if (A.2) does not hold, the flow on such an arc can be fixed to zero. Finally, assumptions (A.3)--(A.5) are without loss of generality. \textcolor{red}{An upper bound on $y_t$ can be obtained directly from the flow balance equalities \eqref{const:flowbalance} by using the upper and lower bounds of the other flow variables that appear in the same constraint. As a result}, the flow values on arcs $t\in E$ cannot exceed the capacities implied by (A.3)--(A.5).

Next, we review the submodular inequalities introduced by \cite{W89} that are valid for any capacitated fixed-charge network flow problem. \textcolor{red}{Furthermore, using the path structure, we provide an $O(|E|+|N|)$ time algorithm to compute their coefficients explicitly.} \textcolor{red}{\st{Then, using the path structure, we obtain the explicit submodular inequality coefficients in $O(|E|+|N|)$ time.}}
\section{Submodular inequalities on paths} \label{sec:spi}
Let $S^+\subseteq E^+$ and $L^-\subseteq E^-$. \cite{W89} shows that the value function of the following optimization problem is submodular:
\begin{subequations} \label{opt:vdefn}
\begin{align}	
 v(S^+, L^-) = \max \quad & \sum_{t\in E} a_t y_t \label{vdefn1}\\
 \text{s. t.} \quad & i_{j-1} - r_{j-1} + y(E_j^+) - y(E_j^-) - i_j + r_j \leq d_j, \quad j \in N, \label{const:flowleq}\\
 &0\leq i_j \leq u_j, \quad j\in N, \label{vdefn2}\\
 &0\leq r_j \leq b_j, \quad j\in N, \label{vdefn3}\\
(\text{F2})\qquad &0\leq y_t \leq c_t, \quad t\in E,\label{vdefn4} \\
 & i_0 = i_n = r_0 = r_n = 0,\label{vdefn6} \\
 &y_t = 0, \quad t\in (E^+\setminus S^+)\cup L^-,\label{vdefn5}
\end{align}
\end{subequations}
where $a_t\in \{0,1\}$ for $t\in E^+$ and $a_t \in \{ 0, -1\}$ for $t\in E^-$. The set of feasible solutions of (F\ref{opt:vdefn}) is represented by $\mathcal{Q}$.

We call the sets $S^+$ and $L^-$ that are used in the definition of $v(S^+,L^-)$ the \emph{objective sets}. For ease of notation, we also represent the objective sets as $C := S^+ \cup L^-$. Following this notation, let $v(C) := v(S^+,L^-)$, $v(C\setminus \{t\}) = v(S^+\setminus\{t\},L^-)$ for $t\in S^+$ and $v(C\setminus \{t\}) = v(S^+, L^-\setminus \{t\})$ for $t\in L^-$. \textcolor{red}{Similarly, let $v(C\cup \{t\})= v(S^+\cup\{t\}, L^-)$, for $t\in S^+$ and $v(C\cup \{t\}) = v(S^+,L^-\cup \{t\})$ for $t\in L^-$.} Moreover, let $$\rho_t(C) = v(C\cup \{t\}) - v(C)$$ be the marginal contribution of adding an arc $t$ to $C$ with respect to the value function $v$. \cite{W89} shows that the following inequalities are valid for $\mathcal{P}$:
\begin{align}
	\sum_{t\in E} a_t y_t + \sum_{t\in C}\rho_t(C\setminus \{t\}) (1-\bar{x}_t) \leq v(C) + \sum_{t\in E\setminus C } \rho_t(\emptyset) \bar{x}_t, \label{ineq:submod1}\\
	\sum_{t \in E} a_t y_t +\sum_{t\in C} \rho_t(E\setminus \{t \})(1-\bar{x}_t) \leq v(C) + \sum_{t\in E\setminus C} \rho_t(C)\bar{x}_t, \label{ineq:submod2}
\end{align}
where the variable $\bar{x}_t$ is defined as $$\bar{x}_t = \begin{cases}
	x_t, \quad t \in E^+ \\
	1-x_t, \quad t \in E^-.
\end{cases}$$
\textcolor{red}{In fact, inequalities \eqref{ineq:submod1} and \eqref{ineq:submod2} are also valid for fixed-charge network flow formulations where the flow balance constraints \eqref{const:flowbalance} are replaced with constraints \eqref{const:flowleq}. However, in this paper, we focus on formulations with flow balance equalities \eqref{const:flowbalance}.}

We refer to submodular inequalities \eqref{ineq:submod1} and \eqref{ineq:submod2} derived for path structures as \emph{path inequalities}.
 In this paper, we consider sets $S^+$ and $L^-$ such that (F\ref{opt:vdefn}) is feasible for all objective sets $C$ and $C\setminus\{t\}$ for all $t\in C$.

 \ignore{
 \begin{itemize}
	\item[]\textbf{(A.6)} The coefficients $\rho_t(C\setminus \{t\})$ and $\rho_t(E\setminus\{t\})$ for all $t\in C$ and $\rho_t(C)$ for $t\in E\setminus C$ are finite.
 \end{itemize}
}
\ignore{
\begin{figure}
\centering
	\includegraphics[scale=1]{Figures/submodularity}
	\caption{Set of arcs $E$, represented by solid directed lines.} \label{fig:setN}
\end{figure}}
\subsection{Equivalence to the maximum flow problem}
Define sets $K^+$ and $K^-$ such that the coefficients of the objective function \eqref{vdefn1} are:
\begin{align} a_t = \begin{cases} 1, \quad t \in K^+ \\
	-1, \quad t \in K^- \\
	0, \quad \text{otherwise},
\end{cases} \label{coef1}
\end{align}
where $S^+\subseteq K^+\subseteq E^+$ and $K^-\subseteq E^- \setminus L^-$. We refer to the sets $K^+$ and $K^-$ as \emph{coefficient sets}. Let the set of arcs with zero coefficients in \eqref{vdefn1} be represented by $\bar{K}^+ = E^+\setminus K^+$ and $\bar{K}^-=E^- \setminus K^-$. Given a selection of coefficients as described in \eqref{coef1}, we claim that (F\ref{opt:vdefn}) can be transformed to a maximum flow problem. \textcolor{red}{We first show this result assuming $d_j\geq 0$ for all $j\in N$. Then, in Appendix \ref{app:equiv}, we show that the nonnegativity of demand is without loss of generality for the derivation of the inequalities.} \textcolor{red}{\st{We first show this result assuming $d_j\geq 0$ for all $j\in N$; however, we show in Appendix A that for the derivation of the inequalities the nonnegativity of demand is without loss of generality through an appropriate transformation.}}

\begin{proposition} \label{prop:zeros}
Let $S^+\subseteq E^+$ and $L^-\subseteq E^-$ be the objective sets in (F\ref{opt:vdefn}) and let $\mathcal{Y}$ be the nonempty set of optimal solutions of (F\ref{opt:vdefn}). If $d_j\geq 0$ for all $j\in N$, then there exists at least one optimal solution $(\textbf{y}^*,\textbf{r}^*,\textbf{i}^*) \in \mathcal{Y}$ such that $y_t^* = 0$ for $t\in \bar{K}^+ \cup K^- \cup L^-$.
\end{proposition}
\begin{proof}
	Observe that $y^*_t = 0$ for all $t\in E^+\setminus S^+$, due to constraints \eqref{vdefn5}. Since $\bar{K}^+\subseteq E^+\setminus S^+$, $y_t^*=0$, for $t\in \bar{K}^+$ from feasibility of (F\ref{opt:vdefn}). Similarly, $y_t^*=0$ for all $t\in L^-$ by constraints \eqref{vdefn5}.
	
	Now suppose that, $y_t^* = \epsilon > 0$ for some $t\in K_j^-$ (i.e., $a_t=-1$ for arc $t$ in (F\ref{opt:vdefn})). Let the slack value at constraint \eqref{const:flowleq} for node $j$ be $$s_j = d_j -\left[i_{j-1}^* - r_{j-1}^* +y^*(E_j^+) - y^*(E_j^-\setminus \{t\}) -y^*_t -i_j^* + r_j^* \right].$$
	
	If $s_j \geq \epsilon$, then decreasing $y_t^*$ by $\epsilon$ both improves the objective function value and \textcolor{red}{conserves the feasibility of flow balance inequality \eqref{const:flowleq} for node $j$, since $s_j - \epsilon \geq 0$.} 
		
	If $s_j < \epsilon$, then decreasing $y_t^*$ by $\epsilon$ violates flow balance inequality since $s_j - \epsilon < 0$. In this case, there must exist a simple directed path $P$ from either the source node $s_N$ or a node $k\in N\setminus \{j\}$ to node $j$ where all arcs have at least a flow of $(\epsilon - s_j)$. This is guaranteed because, $s_j<\epsilon$ implies that, without the outgoing arc $t$, there is more incoming flow to node $j$ than outgoing. Then, notice that decreasing the flow on arc $t$ and all arcs in path $P$ by $\epsilon - s_j$ conserves feasibility. \textcolor{red}{Moreover, the objective function value either remains the same or increases, because decreasing $y_t$ by $\epsilon-s_j$ increases the objective function value by $\epsilon-s_j$ and the decreasing the flow on arcs in $P$ decreases it by at most $\epsilon-s_j$.} At the end of this transformation, the slack value $s_j$ does not change, however; the flow at arc $t$ is now $y_t^*=s_j$ which is equivalent to the first case that is discussed above. As a result, we obtain a new solution to (F\ref{opt:vdefn}) where $y_t^* = 0$ and the objective value is at least as large.
\ignore{	
	An alternative way to prove this proposition is by duality. Let $\mathbf{p}$ be the vector of dual variables of (F\ref{opt:vdefn}). We observe that the dual constraints that correspond to primal variables $y_t$ for $t\in K_j^-$ are $$-p_i + p_k \geq -1,$$ where the dual variable $p_i$ corresponds to the flow balance inequality \eqref{const:flowleq} for node $j$ and the dual variable $p_k$ corresponds to capacity constraint $y_t\leq c_t$. }
\end{proof}	
\begin{proposition} \label{prop:maxflequiv}
If $d_j \geq 0$ for all $j\in N$, then (F\ref{opt:vdefn}) is equivalent to a maximum flow problem from source $s_N$ to sink $t_N$ \textcolor{red}{on graph $G$}.
\end{proposition}
\begin{proof}
\textcolor{red}{At the optimal solution of problem (F\ref{opt:vdefn}) with objective set $(S^+, L^-)$, the decision variables $y_t$, for $t\in (E^+\setminus S^+) \cup K^- \cup L^-$ can be assumed to be zero due to Proposition \ref{prop:zeros} and constraints \eqref{vdefn5}. Then, these variables can be dropped from (F\ref{opt:vdefn}) since the value $v(S^+,L^-)$ does not depend them and formulation (F\ref{opt:vdefn}) reduces to}
\begin{align}
v(S^+,L^-) = \max \big\{ y(S^+):~ &i_{j-1}-r_{j-1}+y(S^+_j) - y(\bar{K}_j^-) - i_j +r_j\leq d_j, ~j\in N, \notag\\ &\eqref{vdefn2}-\eqref{vdefn6} \big\}. \label{opt:vreform}
\end{align}
Now, we reformulate \eqref{opt:vreform} by representing the left hand side of the flow balance constraint by a new nonnegative decision variable $z_j$ that has an upper bound of $d_j$ for each $j\in N$:
\begin{align*}
\max \big\{ y(S^+):~ &i_{j-1}-r_{j-1}+y(S^+_j) - y(\bar{K}_j^-) - i_j +r_j = z_j,~j\in N, \\ &0\leq z_j\leq d_j, ~j\in N, \quad \eqref{vdefn2}-\eqref{vdefn6}\big\}.
\end{align*}
The formulation above is equivalent the maximum flow formulation from the source node $s_N$ to the sink node $t_N$ for the path structures we are considering in this paper.
\end{proof}

Under the assumption that $d_j\geq 0$ for all $j\in N$, Proposition \ref{prop:zeros} and Proposition \ref{prop:maxflequiv} together show that the optimal objective function value $v(S^+,L^-)$ can be computed by solving a maximum flow problem from source $s_N$ to sink $t_N$. We generalize this result in Appendix \ref{app:equiv} for node sets $N$ such that $d_j<0$ for some $j\in N$. As a result, obtaining the explicit coefficients of submodular inequalities \eqref{ineq:submod1} and \eqref{ineq:submod2} reduces to solving $|E|+1$ maximum flow problems. For a general underlying graph, solving $|E|+1$ maximum flow problems would take $O(|E|^2 |N|)$ time (e.g., see \cite{maxflow2}), where $|E|$ and $|N|$ are the number of arcs and nodes, respectively. In the following subsection, by utilizing the equivalence of maximum flow and minimum cuts and the path structure, we show that all coefficients of \eqref{ineq:submod1} and \eqref{ineq:submod2} can be obtained in $O(|E|+|N|)$ time using dynamic programming.

\subsection{Computing the coefficients of the submodular inequalities}
Throughout the paper, we use minimum cut arguments to find the explicit coefficients of inequalities \eqref{ineq:submod1} and \eqref{ineq:submod2}. \textcolor{red}{Figure \ref{fig:mincut1} illustrates an example where $N = [1,5]$, $E^+=[1,5]$, $E^- = [6,10]$ and in Figure \ref{fig:mincut2}, we give an example of an $s_N-t_N$ cut for $S^+  =\{2,4,5\}$, $L^- = \{10\}$ and $\bar{K}^- = \{7, 9,10\}$. The dashed line in Figure \ref{fig:mincut2} represents a cut that corresponds to the partition $\{s_N, 2, 5\}$ and $\{t_N, 1,3,4\}$ with a value of $b_1 + d_2 + c_7 +u_2 + c_4 + b_4+d_5 $. } Moreover, we say that a cut \emph{passes below node $j$} if $j$ is in the source partition and \emph{passes above node $j$} if $j$ is in the sink partition.

\begin{figure}
\centering
	\begin{subfigure}[b]{0.75\textwidth}
	        \centering
	        \includegraphics[scale=1]{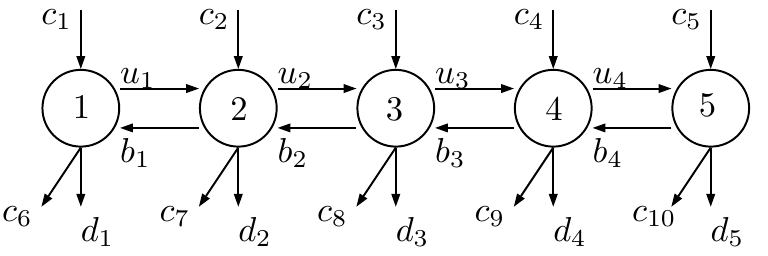}
	        \caption{\textcolor{red}{A path graph with $E^+=[1,5]$ and $E^-=[6,10]$.}} \label{fig:mincut1}
	  \end{subfigure}
	  
	  \vspace{0.4cm} 	    
	  \begin{subfigure}[b]{0.75\textwidth}
	        \centering
	        \includegraphics[scale=1]{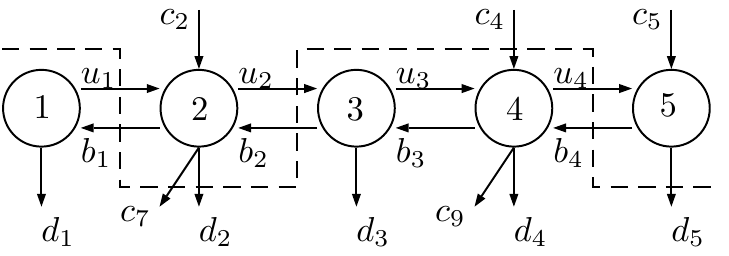}
	        \caption{\textcolor{red}{An $s_N-t_N$ cut for set $S^+=\{2,4,5\}$, $L^-=\{10\}$ and $\bar{K}^-=\{7,9,10\}$.}}\label{fig:mincut2}
	    \end{subfigure}
	\caption{\textcolor{red}{An example of an $s_N-t_N$ cut.}} \label{fig:mincut}
\end{figure}

Let $\alpha_j^u$ and $\alpha_j^d$ be the minimum value of a cut on nodes $[1,j]$ that passes above and below node $j$, respectively. Similarly, let $\beta_j^u$ and $\beta_j^d$ be the minimum values of cuts on nodes $[j,n]$ that passes above and below node $j$ respectively. Finally, let $$S^- = E^- \setminus (K^-\cup L^-),$$ where $K^-$ is defined in \eqref{coef1}. Recall that $S^+$ and $L^-$ are the given objective sets. Given the notation introduced above, all of the arcs in sets $S^-$ and $L^-$ have a coefficient zero in (F\ref{opt:vdefn}). Therefore, dropping an arc from $L^-$ is equivalent to adding that arc to $S^-$. We compute $\alpha_j^{\{u,d\}}$ by a forward recursion and $\beta_j^{\{u,d\}}$ by a backward recursion:
\begin{align}
	&\alpha_j^u = \min \{ \alpha_{j-1}^d + u_{j-1}, \alpha_{j-1}^u \} + c(S_j^+) \label{eqn:alphau1} \\
	&\alpha_j^d = \min \{ \alpha_{j-1}^d, \alpha_{j-1}^u + b_{j-1} \} + d_j + c(S_j^-), \label{eqn:alphad1}
\end{align}
where $\alpha_0^u = \alpha_0^d = 0$ and
\begin{align}
	&\beta_j^u = \min \{ \beta_{j+1}^u, \beta_{j+1}^d + b_{j} \} + c(S_j^+) \label{eqn:betau1}\\
	&\beta_j^d = \min \{ \beta_{j+1}^u + u_{j}, \beta_{j+1}^d \} + d_j + c(S_j^-), \label{eqn:betad1}
\end{align}
where $\beta_{n+1}^u = \beta_{n+1}^d = 0$.

Let $m_j^u$ and $m_j^d$ be the values of minimum cuts for nodes $[1,n]$ that pass above and below node $j$, respectively. Notice that \begin{align} m_j^u = \alpha_j^u + \beta_j^u - c(S_j^+)\label{mincutup} \end{align} and \begin{align}m_j^d = \alpha_j^d + \beta_j^d - d_j - c(S_j^-). \label{mincutdown} \end{align} For convenience, let $$m_j := \min \{ m_j^u , m_j^d\}.$$ Notice that $m_j$ is the minimum of the minimum cut values that passes above and below node $j$. Since the minimum cut corresponding to $v(C)$ has to pass either above or below node $j$, $m_j$ is equal to $v(C)$ for all $j\in N$. As a result, the minimum cut (or maximum flow) value for the objective set $C = S^+ \cup L^-$ is \begin{align} v(C) = m_1 = \dots = m_n. \label{mincutandm}\end{align}

\begin{proposition}
All values $m_j$, for $j\in N$, can be computed in $O(|E|+|N|)$ time.
\end{proposition}

Obtaining the explicit coefficients of inequalities \eqref{ineq:submod1} and \eqref{ineq:submod2} also requires finding $v(C\setminus \{t\})$ for $t\in C$ and $v(C\cup \{t\})$ for $t\notin C$ in addition to $v(C)$. It is important to note that we do not need to solve the recursions above repeatedly. Once the values $m_j^u$ and $m_j^d$ are obtained for the set $C$, the marginals $\rho_t(C\setminus \{t\})$ and $\rho_t(C)$ can be found in $O(1)$ time for each $t\in E$.

We use the following observation while providing the marginal values $\rho_t(C)$ and $\rho_t(C\setminus \{t\})$ in closed form.

\begin{observation} \label{obs:difference}
Let $c\geq 0$ and $d := (b-a)^+$, then,
\begin{enumerate}[1.]
	\item[]1. $\min \{ a+c , b\} - \min \{ a,b \} = \min \{ c, d \},$
	\item[]2. $\min\{a, b\} - \min \{ a, b - c  \} = (c - d)^+$.
\end{enumerate}
\end{observation}
\textcolor{red}{\st{In the remainder of this section, we give the coefficients $\rho_t$ for inequalities (3) and (4) explicitly.}}
\textcolor{red}{In the remainder of this section, we give a linear-time algorithm to compute the coefficients $\rho_t$ for inequalities} \eqref{ineq:submod1} and \eqref{ineq:submod2} explicitly \textcolor{red}{for paths.}

\subsection*{Coefficients of inequality \eqref{ineq:submod1}: Path cover inequalities}
Let $S^+$ and $L^-$ be the objective sets in (F\ref{opt:vdefn}) and $S^-\subseteq E^-\setminus L^-$. We select the coefficient sets in \eqref{coef1} as $K^+ = S^+$ and $K^- = E^-\setminus (L^-\cup S^-)$ to obtain the explicit form of inequality \eqref{ineq:submod1}. As a result, the set definition of $S^-=E^-\setminus (K^-\cup L^-)$ is conserved.
\begin{definition}
	Let the coefficient sets in \eqref{coef1} be selected as above and $(S^+,L^-)$ be the objective set. The set $(S^+,S^-)$ is called a \emph{path cover} for the node set $N$ if $$v(S^+,L^-) = d_{1n} + c(S^-).$$
\end{definition}
For inequality \eqref{ineq:submod1}, we assume that the set $(S^+, S^-)$ is a path cover for $N$. Then, by definition, $$v(C) = m_1 = \dots = m_n = d_{1n} + c(S^-).$$
After obtaining the values $m_j^u$ and $m_j^d$ for a node $j\in N$ using recursions in \eqref{eqn:alphau1}-\eqref{eqn:betad1}, it is trivial to find the minimum cut value after dropping an arc $t$ from $S_j^+$:
$$v(C \setminus \{t\}) = \min \{ m_j^u - c_t, m_j^d \}, \quad t\in S_j^+, \quad  j\in N.$$ Similarly, dropping an arc $t\in L_j^-$ results in the minimum cut value:
$$v(C \setminus \{ t\}) = \min \{ m_j^u, m_j^d + c_t \}, \quad t\in L_j^-, \quad  j\in N.$$
Using Observation \ref{obs:difference}, we obtain the marginal values $$\rho_t(C\setminus \{t\}) = (c_t-\lambda_j)^+, \quad t\in S^+$$ and $$\rho_t(C\setminus \{t\}) = \min\{\lambda_j, c_t\}, \quad t\in L^-$$ where $$\lambda_j = (m_j^u - m_j^d)^+, \quad j\in N.$$

On the other hand, all the coefficients $\rho_t(\emptyset) = 0$ for arcs $t\in E\setminus C$. First, notice that, for $t\in E^+\setminus S^+$, $v(\{t\}) = 0$, because the coefficient $a_t=0$ for $t\in E^+\setminus S^+$. Furthermore, $v(\{t\})=0$ for $t\in E^-\setminus L^-$, since all incoming arcs would be closed for an objective set $(\emptyset, \{t\})$. As a result, inequality \eqref{ineq:submod1} for the objective set $(S^+, L^-)$ can be written as
\begin{align}
	y(S^+) + \sum_{t\in S^+} (c_t - \lambda_j)^+(1-x_t) \leq  d_{1n} + c(S^-) + \sum_{t\in L^-} \min \{ c_t, \lambda_j\}x_t + y(E^-\setminus(L^-\cup S^-)). \label{ineq:submod1exp}
\end{align}
We refer to inequalities  \eqref{ineq:submod1exp} as \emph{path cover inequalities}.
\begin{remark}
Observe that for a path consisting of a single node $N = \{j\}$ with demand $d:=d_j > 0$, the path cover inequalities \eqref{ineq:submod1exp} reduce to the flow cover inequalities \citep{PvRW85,vRW86}. Suppose that the path consists of a single node $N = \{j\}$ with demand $d:=d_j > 0$. Let $S^+\subseteq E^+$ and $S^- \subseteq E^-$. The set $(S^+,S^-)$ is a flow cover if $\lambda := c(S^+) - d-c(S^-) > 0$ and the resulting path cover inequality \begin{align} y(S^+) + \sum_{t\in S^+} (c_t - \lambda)^+ (1-x_t) \leq d + c(S^-) + \lambda x(L^-) + y(E^-\setminus L^{-}) \label{ineq:fc}\end{align}
is a flow cover inequality.
\end{remark}
\begin{proposition} \label{rem:fcpc}
	Let $(S^+,S^-)$ be a path cover for the node set $N$. The path cover inequality for node set $N$ is at least as strong as the flow cover inequality for the single node relaxation obtained by merging the nodes in $N$.
\end{proposition}
\begin{proof}
	Flow cover and path cover inequalities differ in the coefficients of variables $x_t$ for $t\in S^+$ and $t\in L^-$. Therefore, we compare the values $\lambda_j$, $j\in N$ of path cover inequalities \eqref{ineq:submod1exp} to the value $\lambda$ of flow cover inequalities \eqref{ineq:fc} \textcolor{red}{and show that $\lambda_j \leq \lambda$, for all $j\in N$.} The merging of node set $N$ in graph $G$ is equivalent to relaxing the values $u_j$ and $b_j$ to be infinite for $j\in[1,n-1]$. As a result, the value of the minimum cut that goes above a the merged node is $\bar{m}^u = c(S^+)$ and the value of the minimum cut that goes below the merged node is $\bar{m}^d = d_{1n}+c(S^-)$. Now, observe that the recursions in \eqref{eqn:alphau1}-\eqref{eqn:betad1} imply that the minimum cut values for the original graph $G$ are smaller: $$m_j^u = \alpha_j^u+\beta_j^u - c(S_j^+) \leq c(S^+)=\bar{m}^u$$ and $$m_j^d=\alpha_j^d+\beta_j^d - d_j - c(S_j^-) \leq d_{1n}+c(S^-)=\bar{m}^d$$ for all $j\in N$. Recall that the coefficient for the flow cover inequality is $\lambda = (\bar{m}^u-\bar{m}^d)^+$ and the coefficients for path cover inequality are $\lambda_j = (m_j^u-m_j^d)^+$ for $j\in N$. The fact that $(S^+,S^-)$ is a path cover implies that $m_j^d = d_{1n}+c(S^-)$ for all $j\in N$. Since $\bar{m}^d = m_j^d$ and $m_j^u\leq\bar{m}^u$ for all $j\in N$, we observe that  $\lambda_j \leq \lambda$ for all $j\in N$. Consequently, the path cover inequality \eqref{ineq:submod1exp} is at least as strong as the flow cover inequality \eqref{ineq:fc}.
	\end{proof}

\begin{example}
\begin{figure} [h]
\centering
	\includegraphics{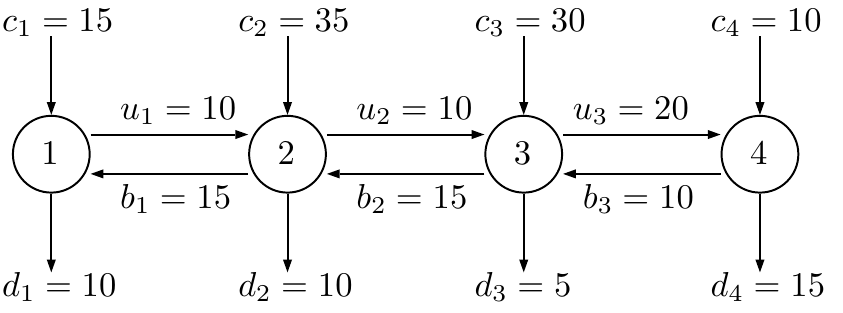}
	\caption{A lot-sizing instance with backlogging.} \label{fig:ex1}
\end{figure}
Consider the lot-sizing instance in Figure \ref{fig:ex1} where $N=[1,4]$, $S^+ = \{2, 3 \}$, $L^- = \emptyset$. Observe that $m_1^u = 45$, $m_1^d = 40$, $m_2^u = 65$, $m_2^d = 40$, $m_3^u = 60$, $m_3^d = 40$, and $m_4^u = 45$, $m_4^d = 40$. Then, $\lambda_1 = 5$, $\lambda_2 = 25$, $\lambda_3 = 20$, and $\lambda_4 = 5$ leading to coefficients $10$ and $10$ for $(1-x_2)$ and $(1-x_3)$, respectively. Furthermore, the maximum flow values are $v(C) = 40$, $v(C\setminus\{2\}) = 30$, and $v(C\setminus\{3\}) = 30$. Then, the resulting path cover inequality \eqref{ineq:submod1exp} is
\begin{align}
	y_2+y_3 + 10(1-x_2) + 10(1-x_3) \leq 40, \label{ex:pc}
\end{align}
and it is facet-defining for conv$(\mathcal{P})$ as will be shown in Section \ref{sec:facet}. Now, consider the relaxation obtained by merging the nodes in $[1,4]$ into a single node with incoming arcs $\{1,2,3,4\}$ and demand $d= 40$. As a result, the flow cover inequalities can be applied to the merged node set. The excess value for the set $S^+ = \{2,3\}$ is $\lambda = c(S^+)-d = 25$. Then, the resulting flow cover inequality \eqref{ineq:fc} is
\begin{align*}
	y_2+y_3 + 10(1-x_2) + 5(1-x_3) \leq 40,
\end{align*}
and it is weaker than the path cover inequality \eqref{ex:pc}.
\qed
\end{example}

\subsection*{Coefficients of inequality \eqref{ineq:submod2}: Path pack inequalities}
Let $S^+$ and $L^-$ be the objective sets in (F\ref{opt:vdefn}) and let $S^-\subseteq E^- \setminus L^-$. We select the coefficient sets in \eqref{coef1} as $K^+ = E^+$ and $K^- = E^-\setminus (S^-\cup L^-)$ to obtain the explicit form of inequality \eqref{ineq:submod2}. As a result, the set definition of $S^-=E^-\setminus (K^-\cup L^-)$ is conserved.

\begin{definition}
	Let the coefficients in \eqref{coef1} be selected as above and $(S^+,L^-)$ be the objective set. The set $(S^+,S^-)$ is called a \emph{path pack} for node set $N$ if $$v(S^+,L^-) = c(S^+).$$
\end{definition}
For inequality \eqref{ineq:submod2}, we assume that the set $(S^+,S^-)$ is a path pack for $N$ and \textcolor{red}{$L^-=\emptyset$ for simplicity}. Now, we need to compute the values of $v(C)$, $v(E)$, $v(E\setminus \{ t \})$ for $t\in C$ and $v(C\cup \{t \})$ for $t\in E\setminus C$. The value of $v(C\cup \{t \})$ can be obtained using the values $m_j^u$ and $m_j^d$ that are given by recursions \eqref{eqn:alphau1}-\eqref{eqn:betad1}. Then,
$$v(C\cup \{ t\}) = \min \{ m_j^u + c_t, m_j^d \}, \quad t\in E_j^+ \setminus S_j^+, \quad  j\in N$$ and
$$v(C\cup \{ t\}) = \min \{ m_j^u, m_j^d + c_t \}, \quad t\in S_j^-, \quad  j\in N.$$
Then, using Observation \ref{obs:difference}, we compute the marginal values $$\rho_t(C) = \min\{c_t,\mu_j\}, \quad t\in E_j^+\setminus S_j^+$$ and $$\rho_t(C) = (c_t-\mu_j)^+, \quad t\in S_j^-,$$ where $$\mu_j = (m_j^d-m_j^u)^+, \quad j\in N.$$

\ignore{
\begin{proposition} \label{prop:rhoS} If (F\ref{opt:vdefn}) is feasible with $L^- = E^-$, then
for the coefficient sets $K^+ = E^+$ and $K^-=E^-\setminus S^-$ for some $S^- \subseteq E^-$, $v(E\setminus \{t\}) = v(E)$ for all $t\in C$.	
\end{proposition}

%\begin{proof}
%Since $L^-=\emptyset$, it does not appear in the inequality and $C= S^+$. %Assumption (A.6) ensures feasibility of (F\ref{opt:vdefn}) when the objective set is $(E^+,E^-)$.
}

Next, we compute the values $v(E)$ and $v(E\setminus \{t\})$ for $t\in C$. The feasibility of (F\ref{opt:ls}) implies that $(E^+,\emptyset)$ is a path cover for $N$. %Now, suppose that an arc $t\in S^+$ is closed to compute $v(E\setminus \{t\})$ in (F\ref{opt:vdefn}).
By Assumption (A.1), $(E^+\setminus\{t\},\emptyset)$ is also a path cover for $N$ for each $t\in S^+$. Then $v(E) = v(E\setminus \{t\})=d_{1n}$
%\end{proof}
%Proposition \ref{prop:rhoS} implies that
and $$\rho_t(E\setminus\{t\})=0, \quad t\in S^+\cup L^-.$$ Then, inequality \eqref{ineq:submod2} can be explicitly written as
\begin{eqnarray}
	y(S^+) + \sum_{t\in E^+\setminus S^+} \left(y_t - \min \{c_t, \mu_j \} x_t \right) + \sum_{t\in S^-} (c_t-\mu_j)^ + (1-x_t) \leq c(S^+)  + y(E^-\setminus S^-). \label{ineq:submod2exp}
\end{eqnarray}
We refer to inequalities \eqref{ineq:submod2exp} as \emph{path pack inequalities}.

\begin{remark}
Observe that for a path consisting of a single node $N = \{j\}$ with demand $d:=d_j > 0$, the path pack inequalities \eqref{ineq:submod2exp}, reduce to the flow pack inequalities (\citep{A:fp}). Let $(S^+, S^-)$ be a flow pack and $\mu := d - c(S^+) + c(S^-)>0$. Moreover, the maximum flow that can be sent through $S^+$ for demand $d$ and arcs in $S^-$ is $c(S^+)$. Then, the value function $v(S^+) = c(S^+)$ and the resulting path pack inequality
\begin{eqnarray}
	y(S^+) + \sum_{t\in E^+\setminus S^+} \left(y_t - \min\{c_t, \mu\}x_t \right) + \sum_{t\in S^-} (c_t-\mu)^+(1-x_t) \leq c(S^+) + y(E^-\setminus S^-) \label{ineq:fp}
\end{eqnarray}
is equivalent to the flow pack inequality.
\end{remark}
\begin{proposition}
	Let $(S^+,S^-)$ be a path pack for the node set $N$. The path pack inequality for $N$ is at least as strong as the flow pack inequality for the single node relaxation obtained by merging the nodes in $N$.
\end{proposition}
\begin{proof}
	The proof is similar to that of Proposition \ref{rem:fcpc}. Flow pack and path pack inequalities only differ in the coefficients of variables $x_t$ for $t\in E^+\setminus S^+$ and $t\in S^-$. Therefore, we compare the values $\mu_j$, $j\in N$ of path pack inequalities \eqref{ineq:submod2exp} to the value $\mu$ of flow pack inequalities \eqref{ineq:fp} \textcolor{red}{and show that $\mu_j \leq \mu$ for all $j\in N$}. For the single node relaxation, the values of the minimum cuts that pass above and below the merged node are $\bar{m}^u = c(S^+)$ and $\bar{m}^d=d_{1n}+c(S^-)$, respectively. The recursions in \eqref{eqn:alphau1}-\eqref{eqn:betad1} imply that $$m_j^u=\alpha_j^u+\beta_j^u - c(S_j^+) \leq c(S^+)=\bar{m}^u$$ and $$m_j^d = \alpha_j^d+\beta_j^d -d_j - c(S_j^-) \leq d_{1n}+c(S^-)=\bar{m}^d.$$
	The coefficient for flow pack inequality is $\mu=(\bar{m}^d-\bar{m}^u)^+$ and for path pack inequality $\mu_j = (m_j^d-m_j^u)^+$. Since $(S^+,S^-)$ is a path pack, the minimum cut passes above all nodes in $N$ and $m_j^u = c(S^+)$ for all $j\in N$. As a result, $m_j^u=\bar{m}^u$ for all $j\in N$ and $m_j^d\leq \bar{m}^d$. Then, observe that the values $$\mu_j\leq \mu, \quad j\in N.$$
\end{proof}
\noindent\emph{Example} 1 (continued).
	Recall the lot-sizing instance with backlogging given in Figure \ref{fig:ex1}. Let the node set $N=[1,4]$ with $E^- = \emptyset$ and $S^+ = \{3\}$. Then, $m_1^u = 30$, $m_1^d = 40$, $m_2^u = 30$, $m_2^d = 40$, $m_3^u=30$, $m_3^d=30$, $m_4^u=30$, $m_4^d=30$, leading to $\mu_1 = 10$, $\mu_2 = 10$, $\mu_3 = 0$ and $\mu_4 = 0$. Moreover, the maximum flow values are $v(C) = 30$, $v(C\cup \{1\}) = 40$, $v(C\cup \{2\}) = 40$, $v(C\cup\{4\}) = 30$, $v(E)=40$, and $v(E\setminus \{3\}) = 40$. Then the resulting path pack inequality \eqref{ineq:submod2exp} is
	\begin{align}
		y_1+y_2+y_3+ y_4 \leq 30 + 10x_1 + 10x_2 \label{ex:pp}
	\end{align}
and it is facet-defining for conv$(\mathcal{P})$ as will be shown in Section \ref{sec:facet}. Now, suppose that the nodes in $[1,4]$ are merged into a single node with incoming arcs $\{1,2,3,4\}$ and demand $d=40$. For the same set $S^+$, we get $\mu = 40-30=10$. Then, the corresponding flow pack inequality \eqref{ineq:fp} is
\begin{align*}
	y_1+y_2+y_3+y_4 \leq 30 + 10 x_1 + 10 x_2 + 10 x_4,
\end{align*}
which is weaker than the path pack inequality \eqref{ex:pp}.
	\qed

\begin{proposition}
If $|E^+\setminus S^+| \leq 1$ and $S^- = \emptyset$, then inequalities \eqref{ineq:submod1exp} and \eqref{ineq:submod2exp} are equivalent.
\end{proposition}
\begin{proof}
If $E^+\setminus S^+ = \emptyset$ and $S^- = \emptyset$, then it is easy to see that the coefficients of inequality \eqref{ineq:submod2exp} are the same as \eqref{ineq:submod1exp}. Moreover, if $|E^+\setminus S^+| = 1$ (and wlog $E^+\setminus S^+ = \{j\}$), then the resulting inequality \eqref{ineq:submod2exp} is
	\begin{align*}
		y(E^+) - y(E^-) & \leq v(C) + \rho_j(C)x_j \\
		& = v(C) + \big( v(C \cup \{j\}) - v(C)\big) x_j \\
		& = v(C \cup \{j\}) - \rho_j(C) (1-x_j),
	\end{align*}
	which is equivalent to path cover inequality \eqref{ineq:submod1exp} with the objective set $(E^+,\emptyset)$.
\end{proof}

\section{The strength of the path cover and pack inequalities} \label{sec:facet}
The capacities of the forward and the backward path arcs play an important role in finding the coefficients of the path cover and pack inequalities \eqref{ineq:submod1exp} and \eqref{ineq:submod2exp}. Recall that $K^+$ and $K^-$ are the coefficient sets in \eqref{coef1}, $(S^+, L^-)$ is the objective set for (F\ref{opt:vdefn}) and $S^- = E^-\setminus (K^-\cup L^-)$. 

\begin{definition}
	A node $j\in N$ is called \emph{backward independent} for set $(S^+,S^-)$ if $$\alpha_j^u = \alpha_{j-1}^d +u_{j-1} + c(S_{j}^+),$$ or $$\alpha_j^d = \alpha_{j-1}^u+b_{j-1} + d_j + c(S_j^-).$$
\end{definition}
\begin{definition}
	A node $j\in N$ is called \emph{forward independent} for set $(S^+,S^-)$ if $$\beta_j^u = \beta_{j+1}^d +b_{j} + c(S_{j}^+),$$ or $$\beta_j^d = \beta_{j+1}^u+ u_{j} + d_j + c(S_j^-).$$
\end{definition}
\textcolor{red}{Intuitively, backward independence of node $j\in N$ implies that the minimum cut either passes through the forward path arc $(j-1,j)$ or through the backward path arc $(j,j-1)$. Similarly, forward independence of node $j\in N$ implies that the minimum cut either passes through the forward path arc $(j,j+1)$ or through the backward path arc $(j+1,j)$.} In Lemmas \ref{obs:revindep} and \ref{obs:fwdindep} below, we further explain how forward and backward independence affect the coefficients of path cover and pack inequalities. First, let $S_{jk}^+ = \cup_{i=j}^k S_{i}^+$, $S_{jk}^- = \cup_{i=j}^k S_{i}^-$ and $L_{jk}^- = \cup_{i=j}^k L_{i}^-$ if $j\leq k$, and $\emptyset$ otherwise.

\begin{lemma} \label{obs:revindep}
If a node $j\in N$ is backward independent for set $(S^+,S^-)$, then the values $\lambda_j$ and $\mu_j$ do not depend on the sets $S_{1j-1}^+$, $S_{1j-1}^-$ and the value $d_{1j-1}$.
\end{lemma}
\begin{proof}
	If a node $j$ is backward independent, then either $\alpha_j^u = \alpha_{j-1}^d +u_{j-1} + c(S_{j}^+)$ or $\alpha_j^d = \alpha_{j-1}^u+b_{j-1} + d_j + c(S_j^-)$. If $\alpha_j^u = \alpha_{j-1}^d +u_{j-1} + c(S_{j}^+)$, then the equality in \eqref{eqn:alphau1} implies $\alpha_{j-1}^d+u_{j-1}\leq \alpha_{j-1}^u$. As a result, the equality in \eqref{eqn:alphad1} gives $\alpha_j^d = \alpha_{j-1}^d + d_j + c(S_j^-)$. Following the definitions in \eqref{mincutup}--\eqref{mincutdown}, the difference $$w_j:=m_j^u - m_j^d$$ is $\beta_j^u-\beta_j^d +u_{j-1}$ which only depends on sets $S_k^+$ and $S_k^-$ for $k\in[j,n]$, the value $d_{jn}$ and the capacity of the forward path arc $(j-1,j)$.
	
	If $\alpha_j^d = \alpha_{j-1}^u+b_{j-1} + d_j + c(S_j^-)$, then the equality in \eqref{eqn:alphad1} implies $\alpha_{j-1}^u+b_{j-1}\leq \alpha_{j-1}^d$. As a result, the equality in \eqref{eqn:alphau1} gives $\alpha_j^u = \alpha_{j-1}^u + c(S_j^+)$. Then, the difference $w_j = \beta_j^u-\beta_j^d -b_{j-1}$ which only depends on sets $S_k^+$ and $S_k^-$ for $k\in[j,n]$, the value $d_{jn}$ and the capacity of the backward path arc $(j,j-1)$.
	
	Since the values $\lambda_j$ and $\mu_j$ are defined as $(w_j)^+$ and $(-w_j)^+$ respectively, the result follows.
\end{proof}

\begin{remark} \label{revlambdas}
Let $w_j := m_j^u-m_j^d$. If a node $j\in N$ is backward independent for a set $(S^+,S^-)$, then we observe the following: (1) If $\alpha_j^u = \alpha_{j-1}^d +u_{j-1} + c(S_{j}^+)$, then $$w_j = \beta_j^u-\beta_j^d+u_{j-1},$$ and (2) if $\alpha_j^d = \alpha_{j-1}^u+b_{j-1} + d_j + c(S_j^-)$, then $$w_j = \beta_j^u - \beta_j^d - b_{j-1}.$$
\end{remark}

\begin{lemma} \label{obs:fwdindep}
If a node $j\in N$ is forward independent for set $(S^+,S^-)$, then the values $\lambda_j$ and $\mu_j$ do not depend on the sets $S_{j+1n}^+$, $S_{j+1n}^-$ and the value $d_{j+1n}$.
\end{lemma}
\begin{proof}
	The forward independence implies either $\beta_j^u = \beta_{j+1}^d + b_j +c(S_j^+)$ and $\beta_j^d = \beta_{j+1}^d + d_j+c(S_j^-)$ or $\beta_j^u = \beta_{j+1}^u + c(S_j^+)$ and $\beta_j^d=\beta_{j+1}^u +u_j + d_j+c(S_j^-)$. Then, the difference $w_j = m_j^u-m_j^d$ is either $\alpha_j^u+\alpha_j^d+b_j$ or $\alpha_j^u+\alpha_j^d-u_j$ and in both cases, it is independent of the sets $S_k^+$, $S_k^-$ for $k\in [1,j-1]$ and the value $d_{1j-1}$.
\end{proof}

\begin{remark} \label{fwdlambdas}
Let $w_j := m_j^u-m_j^d$. If a node $j\in N$ is forward independent for a set $(S^+,S^-)$, then we observe the following: (1) If $\beta_j^u = \beta_{j+1}^d + b_j +c(S_j^+)$, then $$w_j = \alpha_j^u-\alpha_j^d+b_j,$$ and (2) if $\beta_j^d=\beta_{j+1}^u +u_j + d_j+c(S_j^-)$, then $$w_j = \alpha_j^u-\alpha_j^d-u_j.$$
\end{remark}

\begin{corollary} \label{cor:independence}
If a node $j\in N$ is backward independent for set $(S^+,S^-)$, then the values $\lambda_k$ and $\mu_k$ for $k\in[j,n]$ are also independent of the sets $S_{1j-1}^+$, $S_{1j-1}^-$ and the value $d_{1j-1}$. Similarly, if a node $j\in N$ is forward independent for set $(S^+,S^-)$, then the values $\lambda_k$ and $\mu_k$ for $k\in[1,j]$ are also independent of the sets $S_{j+1n}^+$, $S_{j+1n}^-$ and the value $d_{j+1n}$.
\end{corollary}
\begin{proof}
	The proof follows from recursions in \eqref{eqn:alphau1}--\eqref{eqn:betad1}. If a node $j$ is backward independent, we write $\alpha_{j+1}^u$ and $\alpha_{j+1}^d$ in terms of $\alpha_{j-1}^u$ and $\alpha_{j-1}^d$ and observe that the difference $w_{j+1}= m_{j+1}^u-m_{j+1}^d$ does not depend on $\alpha_{j-1}^u$ nor $\alpha_{j-1}^d$ which implies independence of sets $S_{1j-1}^+$, $S_{1j-1}^-$ and the value $d_{1n}$. We can repeat the same argument for $w_j$, $j\in [j+2,n]$ to show independence.
	
	We show the same result for forward independence by writing $\beta_{j-1}^u$ and $\beta_{j-1}^d$ in terms of $\beta_{j+1}^u$ and $\beta_{j+1}^d$, we observe that $w_j$ does not depend on $\beta_{j+1}^u$ nor $\beta_{j+1}^d$. Then, it is clear that $w_{j-1}$ is also independent of the sets $S_{j+1n}^+$, $S_{j+1n}^-$ and the value $d_{j+1n}$. We can repeat the same argument for $w_j$, $j\in [1,j-1]$ to show independence.
\end{proof}

Proving the necessary facet conditions frequently requires a partition of the node set $N$ into two disjoint sets. Suppose, $N$ is partitioned into $N_1 = [1,j-1]$ and $N_2 = [j,n]$ for some $j\in N$. Let $E_{N1}$ and $E_{N2}$ be the set of non-path arcs associated with node sets $N_1$ and $N_2$. We consider the forward and backward path arcs $(j-1,j)$ and $(j,j-1)$ to be in the set of non-path arcs $E_{N1}$ and $E_{N2}$ since the node $j-1 \in N_i$ and $j \notin N_i$ for $i=1,2$. In particular, $E_{N1}^+:=(j,j-1)\cup E_{1j-1}^+$, $E_{N1}^- := (j-1,j)\cup E_{1j-1}^-$ and $E_{N2}^+:=(j-1,j)\cup E_{jn}^+$, $E_{N2}^- := (j,j-1)\cup E_{jn}^-$, where $E_{k\ell}^+$ and $E_{k\ell}^-$ are defined as $\cup_{i=k}^{\ell}E_i^+$ and $\cup_{i=k}^{\ell}E_i^-$ if $k\leq \ell$ respectively, and as the empty set otherwise. Since the path arcs for $N$ do not have associated fixed-charge variables, one can assume that there exists auxiliary binary variables $\tilde{x}_k=1$ for $k\in\{(j-1,j),(j,j-1)\}$. Moreover, we partition the sets $S^+, S^-$ and $L^-$ into $S_{N1}^+ \supseteq S_{1j-1}^+$, $ S_{N1}^-\supseteq S_{1j-1}^-$, $L_{N1}^-:=L_{1j-1}^-$ and $S_{N2}^+\supseteq S_{jn}^+$, $S_{N2}^- \supseteq S_{jn}^-$, $L_{N2}^-:=L_{jn}^-$. Then, let $v_{1}$ and $v_{2}$ be the value functions defined in (F\ref{opt:vdefn}) for the node sets $N_1$ and $N_2$ and the objective sets $(S_{N1}^+,L_{N1}^-)$ and $(S_{N2}^+,L_{N2}^-)$. Moreover, let $\alpha_j^u$, $\alpha_j^d$, $\beta_j^u$ and $\beta_j^d$ be defined for $j\in N$ in recursions \eqref{eqn:alphau1}-\eqref{eqn:betad1} for the set $(S^+,S^-)$ and recall that $S^-=E^-\setminus (K^-\cup L^-)$.
\begin{lemma} \label{prop:separability1}
Let $(S^+, L^-)$ be the objective set for the node set for $N=[1,n]$. If $\alpha_j^u = \alpha_{j-1}^d+u_{j-1}+c(S_j^+)$ or $\beta_{j-1}^u = \beta_{j}^d + b_{j-1} +c(S_{j-1}^+)$, then $$v(S^+,L^-) = v_{1}(S_{N1}^+,L_{N1}^-) + v_{2}(S_{N2}^+,L_{N2}^-),$$ where $N_1=[1,j-1]$, $N_2 = [j,n]$ and the arc sets are $S_{N1}^+ = (j,j-1)\cup S_{1j-1}^+$, $S_{N2}^+ = (j-1,j)\cup S_{jn}^+$, $S_{N1}^- = S_{1j-1}^-$, $S_{N2}^- = S_{jn}^-$.
\end{lemma}
\begin{proof}
See Appendix \ref{app:proof_sep1}.
\end{proof}

\begin{lemma} \label{prop:separability2}
Let $(S^+, L^-)$ be the objective set for the node set for $N=[1,n]$. If $\alpha_j^d = \alpha_{j-1}^u+b_{j-1}+d_{j-1}+c(S_j^-)$ or $\beta_{j-1}^d = \beta_{j}^u + u_{j-1}+ d_{j-1}+c(S_{j-1}^-)$, then $$v(S^+,L^-) = v_{1}(S_{N1}^+,L_{N1}^-) + v_{2}(S_{N2}^+,L_{N2}^-),$$ where $N_1=[1,j-1]$, $N_2 = [j,n]$ and the arc sets are $S_{N1}^+ = S_{1j-1}^+$, $S_{N2}^+ =  S_{jn}^+$, $S_{N1}^- = (j-1,j)\cup S_{1j-1}^-$, $S_{N2}^- = (j,j-1) \cup S_{jn}^-$.
\end{lemma}
\begin{proof}
See Appendix \ref{app:proof_sep2}.
\end{proof}
\begin{lemma} \label{prop:separability3}
	Let $(S^+, L^-)$ be the objective set for the node set for $N=[1,n]$. If $\alpha_j^u = \alpha_{j-1}^d+u_{j-1}+c(S_j^+)$ and $\beta_{j-1}^d = \beta_{j}^u + u_{j-1}+ d_{j-1}+c(S_{j-1}^-)$, then $$v(S^+,L^-) = v_{1}(S_{N1}^+,L_{N1}^-) + v_{2}(S_{N2}^+,L_{N2}^-),$$ where $N_1=[1,j-1]$, $N_2 = [j,n]$ and the arc sets are $S_{N1}^+ = S_{1j-1}^+$, $S_{N2}^+ =  (j-1,j)\cup S_{jn}^+$, $S_{N1}^- = S_{1j-1}^-$, $S_{N2}^- = (j,j-1) \cup S_{jn}^-$.
\end{lemma}
\begin{proof}
See Appendix \ref{app:proof_sep3}.
\end{proof}
\begin{lemma} \label{prop:separability4}
	Let $(S^+, L^-)$ be the objective set for the node set for $N=[1,n]$. If $\alpha_j^d = \alpha_{j-1}^u+b_{j-1}+d_j + c(S_j^-)$ and $\beta_{j-1}^u = \beta_{j}^d +b_{j-1}+ c(S_{j-1}^+)$, then $$v(S^+,L^-) = v_{1}(S_{N1}^+,L_{N1}^-) + v_{2}(S_{N2}^+,L_{N2}^-),$$ where $N_1=[1,j-1]$, $N_2 = [j,n]$ and the arc sets are $S_{N1}^+ = (j,j-1) \cup S_{1j-1}^+$, $S_{N2}^+ =S_{jn}^+$, $S_{N1}^- = (j-1,j)\cup S_{1j-1}^-$, $S_{N2}^- = S_{jn}^-$.
\end{lemma}
\begin{proof}
See Appendix \ref{app:proof_sep4}.
\end{proof}

In the remainder of this section, we give necessary and sufficient conditions for path cover and pack inequalities \eqref{ineq:submod1exp} and \eqref{ineq:submod2exp} to be facet-defining for the convex hull of $\mathcal{P}$.
\begin{theorem}
\label{thm:facetnec1}
Let $N=[1,n]$, and $d_j\geq 0$ for all $j\in N$. If $L^- = \emptyset$ and the set $(S^+,S^-)$ is a path cover for $N$, then the following conditions are necessary for path cover inequality \eqref{ineq:submod1exp} to be facet-defining for conv$(\mathcal{P})$:
\begin{enumerate}[(i)]
	\item \label{i1:nec1} $\rho_t(C\setminus\{t\}) < c_t$, for all $t\in C$,
	\item \label{i1:nec3} $\max_{t\in S^+} \rho_t(C\setminus \{ t \}) > 0$,
	\item \label{i1:nec6} if a node $j\in [2,n]$ is forward independent for set $(S^+,S^-)$, then node $j-1$ is not backward independent for set $(S^+,S^-)$,
	\item \label{i1:nec6.5}if a node $j\in [1,n-1]$ is backward independent for set $(S^+,S^-)$, then node $j+1$ is not forward independent for set $(S^+,S^-)$,
	\item \label{i1:nec7} if $\max_{t\in S_i^+} (c_t-\lambda_i)^+ = 0$ for $i=p,\dots,n$ for some $p\in[2,n]$, then the node $p-1$ is not forward independent for $(S^+,S^-)$,
	\item \label{i1:nec8} if $\max_{t\in S_i^+} (c_t-\lambda_i)^+ = 0$ for $i=1,\dots,q$ for some $q\in [1,n-1]$, then the node $q+1$ is not backward independent for $(S^+,S^-)$.
\end{enumerate}
\end{theorem}
\begin{proof}
\begin{enumerate}[(i)]
	\item If for some $t'\in S^+$, $\rho_{t'}(C\setminus\{t'\}) \ge c_{t'}$, then the path cover inequality with the objective set $S^+\setminus\{t'\}$ summed with $y_{t'}\leq c_{t'}x_{t'}$ results in an inequality at least as strong. Rewriting the path cover inequality using the objective set $S^+$, we obtain
		\begin{align*}
			\sum_{t\in S^+\setminus\{t'\}}(y_t+\rho_t(S^+\setminus \{t\})(1-x_t))+y_{t'} &\le v(S^+) - \rho_{t'}(S^+\setminus\{t'\})(1-x_{t'})+y(E^-\setminus S^-)\\
			&=v(S^+\setminus\{t'\})+\rho_{t'}(S^+\setminus\{t'\})x_{t'}+y(E^-\setminus S^-).
			\end{align*}
		Now, consider summing the path cover inequality for the objective set $S^+\setminus\{t'\}$
		\begin{align*}
			&\sum_{t\in S^+\setminus\{t'\}}(y_t+\rho_t(S^+\setminus\{t,t'\})(1-x_t)) \le v(S\setminus\{t'\})+y(E^-\setminus S^-),
		\end{align*}
		and $y_{t'}\le c_{t'}x_{t'}$. The resulting inequality dominates inequality~\eqref{ineq:submod1} because $\rho_t(S^+\setminus \{t\})\leq \rho_t(S^+\setminus\{t,t'\})$, from the submodularity of the set function $v$. \textcolor{red}{If the assumption of $L^-=\emptyset$ is dropped, this condition extends for arcs $t \in L^-$ as $\rho_t(C\setminus \{t\}) > -c_t$ with a similar proof.} 

	\item If $L^-=\emptyset$ and $\max_{t\in S^+} \rho_t (C\setminus \{ t \}) = 0$, then summing flow balance equalities \eqref{const:flowbalance} for all nodes $j\in N$ gives an inequality at least as strong.

 \item Suppose a node $j$ is forward independent for $(S^+,S^-)$ and the node $j-1$ is backward independent for $(S^+,S^-)$ for some $j\in [2,n]$. Lemmas \ref{prop:separability1}--\ref{prop:separability4} show that the nodes $N$ and the arcs $C=S^+\cup L^-$ can be partitioned into $N_1= [1,j-1]$, $N_2=[j,n]$ and $C_1$, $C_2$ such that the sum of the minimum cut values for $N_1$, $N_2$ is equal to the minimum cut for $N$. From Remarks \ref{revlambdas} and \ref{fwdlambdas} and Corollary \ref{cor:independence}, it is easy to see that $\lambda_i$ for $i\in N$ will not change by the partition procedures described in Lemmas \ref{prop:separability1}--\ref{prop:separability4}. We examine the four cases for node $j-1$ to be forward independent and node $j$ to be backward independent for the set $(S^+,S^-)$.
\begin{enumerate}[(a)]
	\item Suppose $\alpha_j^u = \alpha_{j-1}^d+u_{j-1}+c(S_j^+)$ and $\beta_{j-1}^u = \beta_{j}^d + b_{j-1} +c(S_{j-1}^+)$. Consider the partition procedure described in Lemma \ref{prop:separability1}, where $S_{N1}^+ = (j,j-1)\cup S_{1j-1}^+$, $S_{N2}^+ = (j-1,j)\cup S_{jn}^+$, $S_{N1}^- = S_{1j-1}^-$, $S_{N2}^- = S_{jn}^-$. Then, the path cover inequalities for nodes $N_1$ and $N_2$
\begin{align*}
	r_{j-1} + \sum_{i=1}^{j-1}\sum_{t\in S_i^+}\left(y_t + (c_t-\lambda_i)^+(1-x_t) \right) \leq v_1(S_{N1}^+) + \sum_{i=1}^{j-1}y(E_i^-\setminus S_i^-) + i_{j-1}
\end{align*}
and
\begin{align*}
	i_{j-1} + \sum_{i=j}^{n}\sum_{t\in S_i^+}\left(y_t + (c_t-\lambda_i)^+(1-x_t) \right) \leq v_2(S_{N2}^+) + \sum_{i=j}^{n}y(E_i^-\setminus S_i^-) + r_{j-1}
\end{align*}
summed gives
\begin{align*}
	\sum_{i=1}^{n}\sum_{t\in S_i^+}\left(y_t + (c_t-\lambda_i)^+(1-x_t) \right) \leq v(S^+) + y(E^-\setminus S^-),
\end{align*}
which is the path cover inequality for $N$ with the objective set $S^+$.
	\item Suppose $\alpha_j^d = \alpha_{j-1}^u+b_{j-1}+d_{j-1}+c(S_j^-)$ and $\beta_{j-1}^d = \beta_{j}^u + u_{j-1}+ d_{j-1}+c(S_{j-1}^-)$. Consider the partition described in Lemma \ref{prop:separability2}, where $S_{N1}^+ = S_{1j-1}^+$, $S_{N2}^+ =  S_{jn}^+$, $S_{N1}^- = (j-1,j)\cup S_{1j-1}^-$, $S_{N2}^- = (j,j-1) \cup S_{jn}^-$. The path cover inequalities for nodes $N_1$ and $N_2$
\begin{align*}
	\sum_{i=1}^{j-1}\sum_{t\in S_i^+} \left(y_t + (c_t-\lambda_i)^+(1-x_t) \right) \leq v_1(S_{N1}^+) + \sum_{i=1}^{j-1}y(E_i^-\setminus S_i^-)
\end{align*}
and
\begin{align*}
	\sum_{i=j}^{n}\sum_{t\in S_i^+} \left(y_t + (c_t-\lambda_i)^+(1-x_t) \right) \leq v_2(S_{N2}^+) + \sum_{i=j}^{n}y(E_i^-\setminus S_i^-).
\end{align*}
summed gives the path cover inequality for nodes $N$ and arcs $C$.

	\item Suppose $\alpha_j^u = \alpha_{j-1}^d+u_{j-1}+c(S_j^+)$ and $\beta_{j-1}^d = \beta_{j}^u + u_{j-1}+ d_{j-1}+c(S_{j-1}^-)$. Consider the partition described in Lemma \ref{prop:separability3}, where $S_{N1}^+ = S_{1j-1}^+$, $S_{N2}^+ =  (j-1,j)\cup S_{jn}^+$, $S_{N1}^- = S_{1j-1}^-$, $S_{N2}^- = (j,j-1) \cup S_{jn}^-$. The path cover inequalities for nodes $N_1$ and $N_2$
\begin{align*}
	\sum_{i=1}^{j-1}\sum_{t\in S_i^+} \left(y_t + (c_t-\lambda_i)^+(1-x_t) \right) \leq v_1(S_{N1}^+) + \sum_{i=1}^{j-1}y(E_i^-\setminus S_i^-) + i_{j-1}
\end{align*}
and
\begin{align*}
	i_{j-1}+\sum_{i=j}^{n}\sum_{t\in S_i^+} \left(y_t + (c_t-\lambda_i)^+(1-x_t) \right) \leq v_2(S_{N2}^+) + \sum_{i=j}^{n}y(E_i^-\setminus S_i^-).
\end{align*}
summed gives the path cover inequality for nodes $N$ and arcs $C$.

	\item Suppose $\alpha_j^d = \alpha_{j-1}^u+b_{j-1}+d_j + c(S_j^-)$ and $\beta_{j-1}^u = \beta_{j}^d +b_{j-1}+ c(S_{j-1}^+)$. Consider the partition described in Lemma \ref{prop:separability4}, where $S_{N1}^+ = (j,j-1) \cup S_{1j-1}^+$, $S_{N2}^+ =S_{jn}^+$, $S_{N1}^- = (j-1,j)\cup S_{1j-1}^-$, $S_{N2}^- = S_{jn}^-$. The path cover inequalities for nodes $N_1$ and $N_2$
\begin{align*}
	r_{j-1} + \sum_{i=1}^{j-1}\sum_{t\in S_i^+} \left(y_t + (c_t-\lambda_i)^+(1-x_t) \right) \leq v_1(S_{N1}^+) + \sum_{i=1}^{j-1}y(E_i^-\setminus S_i^-)
\end{align*}
and
\begin{align*}
	\sum_{i=j}^{n}\sum_{t\in S_i^+} \left(y_t + (c_t-\lambda_i)^+(1-x_t) \right) \leq v_2(S_{N2}^+) + \sum_{i=j}^{n}y(E_i^-\setminus S_i^-) + r_{j-1}.
\end{align*}
summed gives the path cover inequality for nodes $N$ and arcs $C$.
\end{enumerate}
\item The same argument for condition \eqref{i1:nec6} above also proves the desired result here.
\item Suppose $(c_t-\lambda_i)^+ = 0$ for all $t\in S_i^+$ and $i\in [p,n]$ and the node $p-1$ is forward independent for some $p\in [2,n]$. Then, we partition the node set $N=[1,n]$ into $N_1 = [1,p-1]$ and $N_2=[p,n]$. We follow Lemma \ref{prop:separability1} if $\beta_{p-1}^u=\beta_p^d+b_{p-1}+c(S_{p-1}^+)$ and follow Lemma \ref{prop:separability2} if $\beta_{p-1}^d = \beta_p^u +u_{p-1}+d_{p-1} + c(S_{p-1}^-)$ to define $S_{N1}^+$, $S_{N1}^-$, $S_{N2}^+$ and $S_{N2}^-$. Remark \ref{fwdlambdas} along with the partition procedure described in Lemma \ref{prop:separability1} or \ref{prop:separability2} implies that $\lambda_i$ will remain unchanged for $i\in N_1$. The path cover inequality for nodes $N$ and arcs $C$ is
\begin{align*}
	y(S^+) + \sum_{i=1}^{p-1} \sum_{t\in S_i^+} (c_t-\lambda_i)^+(1-x_t) \leq v(S^+) + y(E^-\setminus S^-).
\end{align*}
If $\beta_{p-1}^u=\beta_{j}^d+b_{p-1}+c(S_{p-1}^+)$, then the path cover inequality for nodes $N_1$ and arcs $S_{N1}^+$, $S_{N1}^-$ described in Lemma \ref{prop:separability1} is
\begin{align*}
	r_{p-1} + \sum_{i=1}^{p-1} \sum_{t\in S_i^+} \left(y_t + (c_t-\lambda_i)^+(1-x_t)\right) \leq v(S_{N1}^+) + \sum_{i=1}^{p-1} y(E_i^-\setminus S_i^-) + i_{p-1}.
\end{align*}
Moreover, let $\bar{m}_p^u$ and $\bar{m}_p^d$ be the values of minimum cut that goes above and below node $p$ for the node set $N_2$ and arcs $S_{N2}^+$, $S_{N2}^-$ and observe that $$\bar{m}_p^u = \beta_p^u + u_{p-1}\text{ and }\bar{m}_p^d = \beta_p^d.$$ Then, comparing the difference $\bar{\lambda}_p:= (\bar{m}_p^u - \bar{m}_p^d)^+ =(\beta_p^u -\beta_p^d + u_{p-1})^+$ to $\lambda_p = (m_p^u-m_p^d)^+ = (\beta_p^u -\beta_p^d + \alpha_p^u -\alpha_p^d + c(S_p^+)-d_p-c(S_p^-))^+ $, we observe that $\bar{\lambda}_p \geq \lambda_p$ since $\alpha_p^u-\alpha_p^d + c(S_p^+)-d_p-c(S_p^-) \leq u_{p-1}$ from \eqref{eqn:alphau1}--\eqref{eqn:alphad1}. Since $(c_t-\lambda_p)^+=0$, then $(c_t-\bar{\lambda}_p)^+ = 0$ as well. Using the same technique, it is easy to observe that $\bar{\lambda}_i \geq \lambda_i$ for $i\in [p+1,n]$ as well. As a result, the path cover inequality for $N_2$ with sets $S_{N2}^+$, $S_{N2}^-$ is
\begin{align*}
	i_{p-1} + \sum_{i=p}^n y(S_i^+) \leq v(S_{N2}^+) + \sum_{i=p}^n y(E_i^-\setminus S_i^-) + r_{p-1}.
\end{align*}
The path cover inequalities for $N_1$, $S_{N1}^+$, $S_{N1}^-$ and for $N_2$, $S_{N2}^+$, $S_{N2}^-$ summed gives the path cover inequality for $N$, $S^+$, $S^-$.

Similarly, if $\beta_{p-1}^d=\beta_{j}^u+u_{p-1}+d_{p-1}+c(S_{p-1}^-)$, the proof follows very similar to the previous argument using Lemma \ref{prop:separability2}. Letting $\bar{m}_p^u$ and $\bar{m}_p^d$ be the values of minimum cut that goes above and below node $p$ for the node set $N_2$ and arcs $S_{N2}^+$, $S_{N2}^-$, we get $$\bar{m}_p^u = \beta_p^u\text{ and }\bar{m}_p^d + b_{p-1} = \beta_p^d$$ under this case. Now, notice that that $\alpha_p^u-\alpha_p^d + c(S_p^+)-d_p-c(S_p^-) \geq -b_{p-1}$ from \eqref{eqn:alphau1}--\eqref{eqn:alphad1}, which leads to $\bar{\lambda}_p \geq \lambda_p$. Then the proof follows same as above.

\item The proof is similar to that of the necessary condition \eqref{i1:nec7}. We use Lemmas \ref{prop:separability3} and \ref{prop:separability4} and Remark \ref{fwdlambdas} to partition the node set $N$ and arcs $S^+$, $S^-$ into node sets $N_1=[1,q]$ and $N_2=[q+1,n]$ for $q\in [2,n]$ and arcs $S_{N1}^+$, $S_{N1}^-$ and $S_{N2}^+$, $S_{N2}^-$. Next, we check the values of minimum cut that goes above and below node $q$ for the node set $N_1$ and arcs $S_{N1}^+$, $S_{N1}^-$. Then, observing $-u_{q}\leq \beta_{q}^u-\alpha_{q}^d + c(S_q^+)-d_q-c(S_q^-) \leq b_{q}$ from \eqref{eqn:betau1}--\eqref{eqn:betad1}, it is easy to show that the coefficients $x_t$ for $t\in S_{N1}^+$ are equal to zero in the path cover inequality for node set $N_1$. As a result, the path cover inequalities for $N_1$, $S_{N1}^+$, $S_{N1}^-$ and for $N_2$, $S_{N2}^+$, $S_{N2}^-$ summed gives the path cover inequality for $N$, $S^+$, $S^-$.
\end{enumerate}
\end{proof}
\textcolor{red}{\begin{remark}
	If the node set $N$ consists of a single node, then the conditions \eqref{i1:nec1} and \eqref{i1:nec3} of Theorem \ref{thm:facetnec1} reduce to the sufficient facet conditions of flow cover inequalities given in \cite[Theorem 4]{PvRW85} and \cite[Theorem 6]{vRW86}. In this setting, conditions \eqref{i1:nec6}--\eqref{i1:nec8} are no longer relevant.
\end{remark}}

\ignore{
\begin{corollary}
	If $S^+ = E^+$ and $L^- = \emptyset$, then the path cover inequality \eqref{ineq:submod1exp} is not facet defining for $\text{conv}(\mathcal{P})$.
\end{corollary}
\begin{proof}
	Since all incoming arcs are included in $S^+$, $v(E^+) = d_{1n}$. Moreover, due to assumption (A.1), we know that $v(E^+\setminus \{ t \}) = d_{1n}$ for all $t\in E^+$ as well. Then $\rho_t(S^+) = 0$ for all $t\in S^+$ which contradicts the necessary condition \eqref{i1:nec3}.
\end{proof}

\begin{corollary}
	If $|S^+| = 1$ and $L^- = \emptyset$, then the path cover inequality is not facet defining for $\text{conv}(\mathcal{P})$.
\end{corollary}
\begin{proof}
	If there is only one arc in $S^+ = \{t\}$, then $\rho_t(C\setminus \{t\}) = \min \{ c_t, d_{1n} \} = c_t$ due to assumption (A.3). This result contradicts with the necessary condition \eqref{i1:nec1}.
\end{proof}}

\begin{theorem}
Let $N=[1,n]$, $E^- = \emptyset$, $d_j> 0$ and $|E_j^+|= 1$, for all $j\in N$ and let the set $S^+$ be a path cover. The necessary conditions in Theorem \ref{thm:facetnec1} along with \begin{enumerate} [(i)]
\item \label{i1:suf1} $(c_t-\lambda_j)^+ > 0$ for all $t\in S_j^+$, $j\in N$,
\item \label{i1:suf2} $(c_t-\lambda_j)^+ < c(E^+\setminus S^+)$ for all $t\in S_j^+$, $j\in N$
\ignore{\item \label{i1:suf3} there does not exist any forward or backward independent node in $N$ for set $S^+$}
\end{enumerate}
are sufficient for path cover inequality \eqref{ineq:submod1exp} to be facet-defining for $\text{conv}(\mathcal{P})$.
\end{theorem}
\begin{proof}
	Recall that $\text{dim}\big(\text{conv}(\mathcal{P})\big) = 2|E| + n -2$. In this proof, we provide $2|E| + n -2$ affinely independent points that lie on the face $F$
	$$F = \left\{ (\mathbf{x}, \mathbf{y}, \mathbf{i}, \mathbf{r}) \in \mathcal{P}~:~ 	y(S^+) + \sum_{t\in S^+} (c_t - \lambda_j)^+(1-x_t) = d_{1n}  \right\}.$$

First, we provide Algorithm \ref{alg:feasible} which outputs an initial feasible solution $(\bar{\mathbf{x}}, \bar{\mathbf{y}},  \bar{\mathbf{i}}, \bar{\mathbf{r}})$, where all the arcs in $S^+$ have non-zero flow. Let $\bar{d}_j$ be the effective demand on node $j$, that is, the sum of $d_j$ and the minimal amount of flow that needs to be sent from the arcs in $S_j^+$ to ensure $v(S^+) = d_{1n}$. In Algorithm \ref{alg:feasible}, we perform a backward pass and a forward pass on the nodes in $N$. This procedure is carried out to obtain the minimal amounts of flow on the forward and backward path arcs to satisfy the demands. For each node $j\in N$, these minimal outgoing flow values added to the demand $d_j$ give the effective demand $\bar{d}_j$.
	\begin{algorithm} [h]
	\caption{} \label{alg:feasible}
	\begin{algorithmic}
		\STATE \textbf{Initialization:} Let $\bar{d}_j=d_j$ for $j\in N$
		\FOR { $j = (n-1)$ \TO $1$}
		\STATE Let $\Delta = \min \left\{ u_j, \left( \bar{d}_{j+1} - c(S_{j+1}^+) \right)^+  \right\}, $
		\STATE $\bar{d}_j = \bar{d}_j + \Delta$, $\bar{d}_{j+1} = \bar{d}_{j+1} - \Delta,$
		\STATE $\bar{i}_{j} = \Delta.$
		\ENDFOR
		
		\FOR { $j = 2$ \TO $n$ }
		\STATE Let $\Delta = \left( \bar{d}_{j-1} - c(S_{j-1}^+) \right)^+,$
		\STATE $\bar{d}_j = \bar{d}_j + \Delta$, $\bar{d}_{j-1} - \Delta$
		\STATE $\bar{r}_{j-1} = \Delta- \min \{ \Delta , \bar{i}_{j-1} \}$
		\STATE $\bar{i}_{j-1} = \bar{i}_{j-1} - \min \{ \Delta , \bar{i}_{j-1} \}$
		\ENDFOR
		\STATE $\bar{y}_j = \bar{d}_j$, for all $j\in S^+$.
		\STATE $\bar{x}_j = 1$ if $j\in S^+$, $0$ otherwise.
		\STATE $\bar{y}_j = \bar{x}_j = 0$, for all $j\in E^-$.
	\end{algorithmic}
	\end{algorithm}

Algorithm \ref{alg:feasible} ensures that at most one of the path arcs $(j-1,j)$ and $(j,j-1)$ have non-zero flow for all $j\in [2,n]$. Moreover, note that sufficient condition \eqref{i1:suf1} ensures that all the arcs in $S^+$ have nonzero flow. Moreover, for at least one node $i\in N$, it is guaranteed that $c(S_i^+) > \bar{d}_i$. Otherwise, $\rho_t(C) = c_t$ for all $t\in S^+$ which contradicts the necessary condition \eqref{i1:nec1}. Necessary conditions \eqref{i1:nec6} and \eqref{i1:nec6.5} ensure that $\bar{i}_j < u_j$ and $\bar{r}_j < b_j$ for all $j=1,\dots ,n-1$. Let $$e := \arg\max_{i\in N}\{ c(S_i^+) - \bar{d}_i \}$$ be the node with the largest excess capacity. Also let $\mathbf{1}_j$ be the unit vector with $1$ at position $j$.

Next, we give $2|S^+|$ affinely independent points represented by $\bar{w}^t = ( \bar{\mathbf{x}}^t, \bar{\mathbf{y}}^t, \bar{\mathbf{i}}^t, \bar{\mathbf{r}}^t)$ and $\tilde{w}^t = (\tilde{\mathbf{x}}^t, \tilde{\mathbf{y}}^t,  \tilde{\mathbf{i}}^t, \tilde{\mathbf{r}}^t)$ for $t\in S^+$:
\begin{enumerate} [(i)]
	\item Select $\bar{w}^e = (\bar{\mathbf{x}}, \bar{\mathbf{y}},  \bar{\mathbf{i}}, \bar{\mathbf{r}})$ given by Algorithm \ref{alg:feasible}. Let $\varepsilon >0$ be a sufficiently small value. We define $\bar{w}^t$ for $e \neq t\in S^+$ as $\bar{\mathbf{y}}^t = \bar{\mathbf{y}}^e + \varepsilon \mathbf{1}_{e} - \varepsilon \mathbf{1}_t$, $\bar{\mathbf{x}}^t = \bar{\mathbf{x}}^e$. If $t < e$, then $\bar{\mathbf{i}}^t = \bar{\mathbf{i}}^e$ and $\bar{r}_j^t = \bar{r}_j^e$ for $j < t$ and for $t \geq e$, $\bar{r}_j^t = \bar{r}_j^e + \varepsilon$ for $ t\leq j < e$.
	\item In this class of affinely independent solutions, we close the arcs in $S^+$ one at a time and open all the arcs in $E^+\setminus S^+$: $\tilde{\mathbf{x}}^t = \bar{\mathbf{x}} - \mathbf{1}_t + \sum_{j\in E^+\setminus S^+} \mathbf{1}_j$. Next, we send an additional $\bar{y}_t - (c_t - \lambda_j)^+$ amount of flow from the arcs in $S^+\setminus \{t \}$. This is a feasible operation because $v(C\setminus \{t\}) = d_{1n} - (c_t-\lambda_j)^+$. Let $(\mathbf{y}^{*}, \mathbf{i}^{*}, \mathbf{r}^{*})$ be the optimal solution of (F\ref{opt:vdefn}) corresponding to $v(S^+\setminus \{t \})$. Then let, $\tilde{y}_j^t = y_j^{*}$ for $j\in S^+\setminus \{t\}$. Since $v(C\setminus \{t \}) < d_{1n}$, additional flow must be sent through nodes in $E^+ \setminus S^+$ to satisfy flow balance equations \eqref{const:flowbalance}. This is also a feasible operation, because of assumption (A.1). Then, the forward and backward path flows $\tilde{\mathbf{i}}^t$ and $\tilde{\mathbf{r}}^t$ are calculated using the flow balance equations.
\end{enumerate}

In the next set of solutions, we give $2|E^+\setminus S^+|-1$ affinely independent points represented by $\hat{w}^t = (\hat{\mathbf{x}}^t, \hat{\mathbf{y}}^t,  \hat{\mathbf{i}}^t, \hat{\mathbf{r}}^t)$ and $\check{w}^t = (\check{\mathbf{x}}^t, \check{\mathbf{y}}^t,  \check{\mathbf{i}}^t, \check{\mathbf{r}}^t)$ for $t\in E^+\setminus S^+$.

\begin{enumerate}[(i)] \setcounter{enumi}{2}
	\item Starting with solution $\bar{w}^e$, we open arcs in $E^+\setminus S^+$, one by one. $\hat{\mathbf{y}}^t = \bar{\mathbf{y}}^e$, $\hat{\mathbf{x}}^t = \bar{\mathbf{x}}^e + \mathbf{1}_t$, $\hat{\mathbf{i}}^t = \bar{\mathbf{i}}^e$, $\hat{\mathbf{r}}^t = \bar{\mathbf{r}}^e$.
	\item If $|E^+\setminus S^+| \geq 2$, then we can send a sufficiently small $\varepsilon > 0$ amount of flow from arc $t\in E^+\setminus S^+$ to $t\neq k \in E^+\setminus S^+$. Let this set of affinely independent points be represented by $\check{w}^t$ for $t\in E^+\setminus S^+$. While generating $\check{w}^t$, we start with the solution $\tilde{w}^e$, where the non-path arc in $S_e^+$ is closed. The feasibility of this operation is guaranteed by the sufficiency conditions \eqref{i1:suf2} and necessary conditions \eqref{i1:nec6} and \eqref{i1:nec6.5}.
	\begin{enumerate}[(a)]
			\item If $\tilde{y}_t^e = c_t$, then there exists at least one arc $t \neq m \in E^+\setminus S^+$ such that $0 \leq \tilde{y}_m^e < c_m$ due to sufficiency assumption \eqref{i1:suf2}, then for each $t \in E^+\setminus S^+$ such that $\tilde{y}_t^e = c_t$, let $\check{\mathbf{y}}^t = \tilde{\mathbf{y}}^e - \varepsilon \mathbf{1}_t + \varepsilon \mathbf{1}_m$, $\check{\mathbf{x}}^t = \tilde{\mathbf{x}}^e$. If $t < m$, then $\check{\mathbf{i}}^t = \tilde{\mathbf{i}}^e$ and $\check{\mathbf{r}}^t = \tilde{\mathbf{r}}^e + \varepsilon \sum_{i=t}^{m-1} \mathbf{1}_i$. If $t > m$, then $\check{\mathbf{i}}^t = \tilde{\mathbf{i}}^e + \varepsilon \sum_{i=m}^{t-1} \mathbf{1}_i $ and $\check{\mathbf{r}}^t = \tilde{\mathbf{r}}^e$.
			\item If $\tilde{y}_t^e < c_t$ and there exists at least one arc $t\neq m \in E^+\setminus S^+$ such that $\tilde{y}_m^e = 0$, then the same point described in (a) is feasible.
			\item If $\tilde{y}_t^e < c_t$ and there exists at least one arc $t\neq m \in E^+\setminus S^+$ such that $\tilde{y}_m^e = c_m$, then, we send $\varepsilon$ amount of flow from $t$ to $m$, $\check{\mathbf{y}}^t = \tilde{\mathbf{y}}^e + \varepsilon \mathbf{1}_t- \varepsilon \mathbf{1}_m$, $\check{\mathbf{x}}^t = \tilde{\mathbf{x}}^e$. If $t < m$, then $\check{\mathbf{i}}^t = \tilde{\mathbf{i}}^e + \varepsilon \sum_{i=t}^{m-1} \mathbf{1}_i$ and $\check{\mathbf{r}}^t = \tilde{\mathbf{r}}^e$. If $t > m$, then $\check{\mathbf{i}}^t = \tilde{\mathbf{i}}^e $ and $\check{\mathbf{r}}^t = \tilde{\mathbf{r}}^e + \varepsilon \sum_{i=m}^{t-1} \mathbf{1}_i$.
		\end{enumerate}
	\end{enumerate}
Finally, we give $n-1$ points that perturb the flow on the forward path arcs $(j,j+1)$ for $j=1,\dots, n-1$ represented by $\breve{w}^j = (\breve{\mathbf{x}}^j, \breve{\mathbf{y}}^j,  \breve{\mathbf{i}}^j, \breve{\mathbf{r}}^j)$. Let $k = \min \{ i\in N: S_i^+ \neq \emptyset \}$ and $\ell = \max\{ i\in N:~ S_i^+\neq \emptyset \}$. The solution given by Algorithm \ref{alg:feasible} guarantees $\bar{i}_j < u_j$ and $\bar{r}_j < b_j$ for $j=1, \dots, n-1$ due to necessary conditions \eqref{i1:nec6} and \eqref{i1:nec6.5}.
\begin{enumerate}[(i)] \setcounter{enumi}{4}
	\item For $j=1, \dots, n-1$, we send an additional $\varepsilon$ amount of flow from the forward path arc $(j,j+1)$ and the backward path arc $(j+1,j)$. Formally, the solution $\breve{w}^j$ can be obtained by: $\breve{\mathbf{y}}^j = \bar{\mathbf{y}}^e$, $\breve{\mathbf{x}}^j = \bar{\mathbf{x}}^e$, $\breve{\mathbf{i}}^j = \bar{\mathbf{i}}^e + \varepsilon \mathbf{1}_j$ and $\breve{\mathbf{r}}^j = \bar{\mathbf{r}}^e + \varepsilon \mathbf{1}_j$.
\ignore{
	\item For $j=1,\dots ,k-1$, starting with solution $\tilde{w}^k$ we send additional $\varepsilon$ flow from both forward and backward path arcs $(j,j+1)$ and $(j+1,j)$. The solution $\breve{w}^j$ can be obtained by: $\breve{\mathbf{y}}^j = \tilde{\mathbf{y}}^k$, $\breve{\mathbf{x}}^j = \tilde{\mathbf{x}}^k$, $\breve{\mathbf{i}}^j = \tilde{\mathbf{i}}^k + \varepsilon \mathbf{1}_j$ and $\breve{\mathbf{r}}^j = \tilde{\mathbf{r}}^k + \varepsilon \mathbf{1}_j$.

	\item For $j=\ell,\dots ,n$, starting with solution $\tilde{w}^{\ell}$ we send additional $\varepsilon$ flow from both forward and backward path arcs $(j,j+1)$ and $(j+1,j)$. The solution $\breve{w}^j$ can be obtained by: $\breve{\mathbf{y}}^j = \tilde{\mathbf{y}}^{\ell}$, $\breve{\mathbf{x}}^j = \tilde{\mathbf{x}}^{\ell}$, $\breve{\mathbf{i}}^j = \tilde{\mathbf{i}}^{\ell} + \varepsilon \mathbf{1}_j$ and $\breve{\mathbf{r}}^j = \tilde{\mathbf{r}}^{\ell} + \varepsilon \mathbf{1}_j$.}
\end{enumerate}
\end{proof}

Next, we identify conditions under which path pack inequality \eqref{ineq:submod2exp} is facet-defining for conv$(\mathcal{P})$.

\begin{theorem} \label{thm:facetnec2}
	Let $N=[1,n]$, $d_j\geq 0$ for all $j\in N$, let the set $(S^+,S^-)$ be a path pack and $L^-=\emptyset$. The following conditions are necessary for path pack inequality \eqref{ineq:submod2exp} to be facet-defining for $\text{conv}(\mathcal{P})$:
	\begin{enumerate} [(i)]
		\item \label{i2:nec1} $\rho_j(S^+) < c_j$, for all $j\in E^+\setminus S^+$,
		\item \label{i2:nec2} $\max_{t\in S^-} \rho_t(C) > 0$,
	\item  \label{i2:nec3}if a node $j\in [2,n]$ is forward independent for set $(S^+,S^-)$, then node $j-1$ is not backward independent for set $(S^+,S^-)$,
	\item \label{i2:nec4}if a node $j\in [1,n-1]$ is backward independent for set $(S^+,S^-)$, then node $j+1$ is not forward independent for set $(S^+,S^-)$,		
	\item \label{i2:nec5}if $\max_{t\in E_i^+\setminus S_i^+} \rho_t(C)= 0$ and $\max_{t\in S_i^-} \rho_t(C)= 0$ for $i=p,\dots,n$ for some $p\in [2,n]$, then the node $p-1$ is not forward independent for $(S^+,S^-)$,
	\item \label{i2:nec6}if $\max_{t\in E_i^+\setminus S_i^+} \rho_t(C)= 0$ and $\max_{t\in S_i^-} \rho_t(C)= 0$ for $i=1,\dots,q$ for some $q\in [1,n-1]$, then the node $q+1$ is not backward independent for $(S^+,S^-)$.	
	\end{enumerate}
\end{theorem}
\begin{proof}
	\begin{enumerate} [(i)]
		\item Suppose that for some $k\in E^+\setminus S^+$, $\rho_k(S^+) = c_k$. Then, recall the implicit form of path pack inequality \eqref{ineq:submod2exp} is
		\begin{align*}
			y(E^+\setminus \{k\}) + y_k + \sum_{t\in S^-} \rho_t(S^+)(1-x_t) \leq v(S^+) + \sum_{k\neq t\in E^+\setminus S^+} \rho_t(S^+)x_t + c_kx_k + y(E^-\setminus S^-).
		\end{align*}
		Now, if we select $a_k=0$ in (F\ref{opt:vdefn}), then the coefficients of $x_k$ and $y_k$ become zero and summing the path cover inequality
		\begin{align*}
			y(E^+\setminus \{k\}) + \sum_{t\in S^-} \rho_t(S^+)(1-x_t) \leq v(S^+) + \sum_{k\neq t\in E^+\setminus S^+} \rho_t(S^+)x_t + y(E^-\setminus S^-).
		\end{align*}		
		with $y_k \leq c_k x_k$ gives the first path cover inequality.
		\item Suppose that $\rho_j (S^+)=0$ for all $j\in S^-$. Then the path pack inequality is
		\begin{align*}
			y(E^+) \leq v(S^+) + \sum_{t\in E^+\setminus S^+} \rho_t(S^+)x_t +y\left(E^-\setminus (L^-\cup S^-)\right),
		\end{align*}
where $L^-=\emptyset$. If an arc $j$ is dropped from $S^-$ and added to $L^-$, then $v(S^+) = v(S^+\cup \{j\})$ since $\rho_j(S^+)=0$ for $j\in S^-$. Consequently, the path pack inequality with $S^- = S^- \setminus \{ j\}$ and $L^- = \{j\}$
		\begin{align*}
	y(E^+) +\sum_{t\in S^-} \rho_t(S^+\cup \{j\})(1-x_t) & \leq v(S^+)+ \sum_{t\in E^+\setminus S^+} \rho_t(S^+\cup \{j\})x_t+y\left(E^-\setminus (L^-\cup S^-)\right).
		\end{align*}
	But since $0\leq \rho_t(S^+\cup \{j\}) \leq \rho_t(S^+)$ from submodularity of $v$ and $\rho_t(S^+)=0$ for all $t\in S^-$, we observe that the path pack inequality above reduces to
	\begin{align*}
	y(E^+) \leq v(S^+)+ \sum_{t\in E^+\setminus S^+} \rho_t(S^+\cup \{j\})x_t+y\left(E^-\setminus (L^-\cup S^-)\right)
	\end{align*}
	and it is at least as strong as the first pack inequality for $S^+$, $S^-$ and $L^-= \emptyset$.
	\end{enumerate}
	\begin{itemize}
	\item[\eqref{i2:nec3}--\eqref{i2:nec4}] We repeat the same argument of the proof of condition \eqref{i1:nec6} of Theorem \ref{thm:facetnec1}. Suppose a node $j$ is forward independent for $(S^+,S^-)$ and the node $j-1$ is backward independent for $(S^+,S^-)$ for some $j\in [2,n]$. Lemmas \ref{prop:separability1}--\ref{prop:separability4} show that the nodes $N$ and the arcs $C=S^+\cup L^-$ can be partitioned into $N_1= [1,j-1]$, $N_2=[j,n]$ and $C_1$, $C_2$ such that the sum of the minimum cut values for $N_1$, $N_2$ is equal to the minimum cut for $N$. From Remarks \ref{revlambdas} and \ref{fwdlambdas} and Corollary \ref{cor:independence}, it is easy to see that $\mu_i$ for $i\in N$ will not change by the partition procedures described in Lemmas \ref{prop:separability1}--\ref{prop:separability4}. We examine the four cases for node $j-1$ to be forward independent and node $j$ to be backward independent for the set $(S^+,S^-)$. For ease of notation, let $$Q_{jk}^+ := \sum_{i=j}^k \sum_{t\in E_i^+\setminus S_i^+} (y_t - \min \{\mu_i, c_t\}x_t)$$ and $$Q_{jk}^- := \sum_{i=j}^k \sum_{t\in S_i^-} (c_t-\mu_i)^+(1-x_t)$$ for $j\leq k$ and $j\in N$, $k\in N$ (and zero if $j>k$), where the values $\mu_i$ are the coefficients that appear in the path pack inequality \eqref{ineq:submod2exp}. As a result, the path pack inequality can be written as
	\begin{align}
		y(S^+) + Q_{1n}^+ \leq v(C) + Q_{1n}^- + y(E^-\setminus S^-). \label{pack0}
	\end{align}
\begin{enumerate}[(a)]
	\item Suppose $\alpha_j^u = \alpha_{j-1}^d+u_{j-1}+c(S_j^+)$ and $\beta_{j-1}^u = \beta_{j}^d + b_{j-1} +c(S_{j-1}^+)$. Consider the partition procedure described in Lemma \ref{prop:separability1}, where $S_{N1}^+ = (j,j-1)\cup S_{1j-1}^+$, $S_{N2}^+ = (j-1,j)\cup S_{jn}^+$, $S_{N1}^- = S_{1j-1}^-$, $S_{N2}^- = S_{jn}^-$. Then, the path pack inequalities for nodes $N_1$ is
\begin{align}
	r_{j-1} + y(S_{1j-1}^+) + Q_{1j-1}^+ \leq v_1(C_1) + Q_{1j-1}^- + y(E_{1j-1}^-\setminus S_{1j-1}^-)f + i_{j-1}. \label{pack3}
\end{align}
	Similarly, the path pack inequality for $N_2$ is
	\begin{align}
	i_{j-1} + y(S_{jn}^+) + Q_{jn}^+ \leq v_2(C_2) + Q_{jn}^- + y(E_{jn}^-\setminus S_{jn}^-) + r_{j-1}. \label{pack4}
\end{align}
	Inequalities \eqref{pack3}--\eqref{pack4} summed gives the path pack inequality \eqref{pack0}.
	\item Suppose $\alpha_j^d = \alpha_{j-1}^u+b_{j-1}+d_{j}+c(S_j^-)$ and $\beta_{j-1}^d = \beta_{j}^u + u_{j-1}+ d_{j-1}+c(S_{j-1}^-)$. Consider the partition described in Lemma \ref{prop:separability2}, where $S_{N1}^+ = S_{1j-1}^+$, $S_{N2}^+ =  S_{jn}^+$, $S_{N1}^- = (j-1,j)\cup S_{1j-1}^-$, $S_{N2}^- = (j,j-1) \cup S_{jn}^-$. The submodular inequality \eqref{ineq:submod2} for nodes $N_1$ where the objective coefficients of (F\ref{opt:vdefn}) are selected as $a_t=1$ for $t\in E_{1j-1}^+$, $a_t=0$ for $t=(j,j-1)$, $a_t = -1$ for $t\in E_{N1}^-\setminus S_{N1}^-$ and $a_t = 0$ for $t\in S_{N1}^-$ is
\begin{align}
	y(S_{1j-1}^+) + \sum_{t\in S_{N1}^+}k_t(1-x_t) + Q_{1j-1}^+ \leq v_1(C_1) - Q_{1j-1}^- + y(E_{1j-1}^-\setminus S_{1j-1}^-), \label{pack1}
\end{align}
where $k_t$ for $t\in S_{N1}^+$ are some nonnegative coefficients. Similarly, the submodular inequality \eqref{ineq:submod2} for nodes $N_2$, where the objective coefficients of (F\ref{opt:vdefn}) are selected as $a_t=1$ for $t\in E_{jn}^+$, $a_t=0$ for $t=(j-1,j)$, $a_t = -1$ for $t\in E_{N2}^-\setminus S_{N2}^-$ and $a_t = 0$ for $t\in S_{N2}^-$ is
\begin{align}
	y(S_{jn}^+) + \sum_{t\in S_{N2}^+}k_t(1-x_t) + Q_{jn}^+ \leq v_2(C_2) - Q_{jn}^- +  y(E_{jn}^-\setminus S_{jn}^-), \label{pack2}
\end{align}
where $k_t$ for $t\in S_{N2}^+$ are some nonnegative coefficients. The sum of inequalities \eqref{pack1}--\eqref{pack2} is at least as strong as the path pack inequality \eqref{pack0}.

	\item Suppose $\alpha_j^u = \alpha_{j-1}^d+u_{j-1}+c(S_j^+)$ and $\beta_{j-1}^d = \beta_{j}^u + u_{j-1}+ d_{j-1}+c(S_{j-1}^-)$. Consider the partition described in Lemma \ref{prop:separability3}, where $S_{N1}^+ = S_{1j-1}^+$, $S_{N2}^+ = (j-1,j)\cup S_{jn}^+$, $S_{N1}^- = S_{1j-1}^-$, $S_{N2}^- = (j,j-1) \cup S_{jn}^-$. The submodular inequality \eqref{ineq:submod2} for nodes $N_1$ where the objective coefficients of (F\ref{opt:vdefn}) are selected as $a_t=1$ for $t\in E_{1j-1}^+$, $a_t=0$ for $t=(j,j-1)$, $a_t = -1$ for $t\in E_{N1}^-\setminus S_{N1}^-$ and $a_t = 0$ for $t\in S_{N1}^-$ is
\begin{align}
	y(S_{1j-1}^+) + \sum_{t\in S_{N1}^+}k_t(1-x_t) + Q_{1j-1}^+ \leq v_1(C_1) - Q_{1j-1}^- + y(E_{1j-1}^-\setminus S_{1j-1}^-) + i_{j-1}, \label{pack5}
\end{align}
where $k_t$ for $t\in S_{N1}^+$ are some nonnegative coefficients. The path pack inequality for $N_2$ is
	\begin{align}
		i_{j-1}+ y(S_{jn}^+) + Q_{jn}^+ \leq v_2(C_2) + Q_{jn}^- + y(E_{jn}^-\setminus S_{jn}^-). \label{pack6}
	\end{align}
	The sum of inequalities \eqref{pack5}--\eqref{pack6} is at least as strong as inequality \eqref{pack0}.
	\item Suppose $\alpha_j^d = \alpha_{j-1}^u+b_{j-1}+d_j + c(S_j^-)$ and $\beta_{j-1}^u = \beta_{j}^d +b_{j-1}+ c(S_{j-1}^+)$. Consider the partition described in Lemma \ref{prop:separability4}, where $S_{N1}^+ = (j,j-1) \cup S_{1j-1}^+$, $S_{N2}^+ =S_{jn}^+$, $S_{N1}^- = (j-1,j)\cup S_{1j-1}^-$, $S_{N2}^- = S_{jn}^-$. The path pack inequalities for nodes $N_1$ is
	\begin{align}
		r_{j-1}+ y(S_{1j-1}^+) + Q_{1j-1}^+ \leq v_1(C_1) + Q_{1j-1}^- + y(E_{1j-1}^-\setminus S_{1j-1}^-). \label{pack7}
	\end{align}
	The submodular inequality \eqref{ineq:submod2} for nodes $N_2$ where the objective coefficients of (F\ref{opt:vdefn}) are selected as $a_t=1$ for $t\in E_{jn}^+$, $a_t=0$ for $t=(j-1,j)$, $a_t = -1$ for $t\in E_{N2}^-\setminus S_{N2}^-$ and $a_t = 0$ for $t\in S_{N2}^-$ is
\begin{align}
	y(S_{jn}^+) + \sum_{t\in S_{N2}^+}k_t(1-x_t) + Q_{jn}^+ \leq v_2(C_2) - Q_{jn}^- + y(E_{jn}^-\setminus S_{jn}^-) + r_{j-1}, \label{pack8}
\end{align}
where $k_t$ for $t\in S_{N2}^+$ are some nonnegative coefficients. The sum of inequalities \eqref{pack7}--\eqref{pack8} is at least as strong as the path pack inequality \eqref{pack0}.
\end{enumerate}
	\item[\eqref{i2:nec5}] Suppose $(c_t-\mu_i)^+ = 0$ for all $t\in S_i^-$ and $i\in [p,n]$ and the node $p-1$ is forward independent. Then, we partition the node set $N=[1,n]$ into $N_1 = [1,p-1]$ and $N_2=[p,n]$. We follow Lemma \ref{prop:separability1} if $\beta_{p-1}^u=\beta_p^d+b_{p-1}+c(S_{p-1}^+)$ and follow Lemma \ref{prop:separability2} if $\beta_{p-1}^d = \beta_p^u +u_{p-1}+d_{p-1} + c(S_{p-1}^-)$ to define $S_{N1}^+$, $S_{N1}^-$, $S_{N2}^+$ and $S_{N2}^-$. Remark \ref{fwdlambdas} along with the partition procedure described in Lemma \ref{prop:separability1} or \ref{prop:separability2} implies that $\mu_i$ will remain unchanged for $i\in N_1$.
	
	If $\beta_{p-1}^u=\beta_{j}^d+b_{p-1}+c(S_{p-1}^+)$, then the coefficients $\mu_i$ of the path pack inequality for nodes $N_1$ and arcs $S_{N1}^+$, $S_{N1}^-$ described in Lemma \ref{prop:separability1} is the same as the coefficients of the path pack inequality for nodes $N$ and arcs $S^+$, $S^-$. Moreover, let $\bar{m}_p^u$ and $\bar{m}_p^d$ be the values of minimum cut that goes above and below node $p$ for the node set $N_2$ and arcs $S_{N2}^+$, $S_{N2}^-$ and observe that $$\bar{m}_p^u = \beta_p^u + u_{p-1}\text{ and }\bar{m}_p^d = \beta_p^d.$$ Then, comparing the difference $\bar{\mu}_p:= (\bar{m}_p^d - \bar{m}_p^u)^+ =(\beta_p^d -\beta_p^u - u_{p-1})^+$ to $\mu_p = (m_p^d-m_p^u)^+ = (\beta_p^d -\beta_p^u + \alpha_p^d -\alpha_p^u - c(S_p^+)+d_p+c(S_p^-))^+ $, we observe that $\bar{\mu}_p \geq \mu_p$ since $\alpha_p^d-\alpha_p^u - c(S_p^+)+d_p+c(S_p^-) \geq -u_{p-1}$ from \eqref{eqn:alphau1}--\eqref{eqn:alphad1}. Since $(c_t-\mu_p)^+=0$, then $(c_t-\bar{\mu}_p)^+ = 0$ as well. Using the same technique, it is easy to observe that $\bar{\mu}_i \geq \mu_i$ for $i\in [p+1,n]$ as well. As a result, the path pack inequality for $N_2$ with sets $S_{N2}^+$, $S_{N2}^-$, summed with the path pack inequality for nodes $N_1$ and arcs $S_{N1}^+$, $S_{N1}^-$ give the path pack inequality for nodes $N$ and arc $S^+$, $S^-$.

Similarly, if $\beta_{p-1}^d=\beta_{j}^u+u_{p-1}+d_{p-1}+c(S_{p-1}^-)$, the proof follows very similar to the previous argument using Lemma \ref{prop:separability2}. Letting $\bar{m}_p^u$ and $\bar{m}_p^d$ be the values of minimum cut that goes above and below node $p$ for the node set $N_2$ and arcs $S_{N2}^+$, $S_{N2}^-$, we get $$\bar{m}_p^u = \beta_p^u\text{ and }\bar{m}_p^d + b_{p-1} = \beta_p^d$$ under this case. Now, notice that that $\alpha_p^d-\alpha_p^u - c(S_p^+)+d_p+c(S_p^-) \leq b_{p-1}$ from \eqref{eqn:alphau1}--\eqref{eqn:alphad1}, which leads to $\bar{\mu}_p \geq \mu_p$. Then the proof follows same as above.

	\item[\eqref{i2:nec6}] The proof is similar to that of the necessary condition \eqref{i2:nec5} above. We use Lemmas \ref{prop:separability3}--\ref{prop:separability4} and Remark \ref{fwdlambdas} to partition the node set $N$ and arcs $S^+$, $S^-$ into node sets $N_1=[1,q]$ and $N_2=[q+1,n]$ and arcs $S_{N1}^+$, $S_{N1}^-$ and $S_{N2}^+$, $S_{N2}^-$. Next, we check the values of minimum cut that goes above and below node $q$ for the node set $N_1$ and arcs $S_{N1}^+$, $S_{N1}^-$. Then, observing $-b_{q}\leq \beta_{q}^d-\alpha_{q}^u - c(S_q^+)+d_q+c(S_q^-) \leq u_{q}$ from \eqref{eqn:betau1}--\eqref{eqn:betad1}, it is easy to see that the coefficients of $x_t$ for $t\in S_{N1}^-$ and $t\in E_{N1}^+\setminus S_{N1}^+$ are equal to zero in the path pack inequality for node set $N_1$. As a result, the path pack inequalities for $N_1$, $S_{N1}^+$, $S_{N1}^-$ and for $N_2$, $S_{N2}^+$, $S_{N2}^-$ summed gives the path pack inequality for $N$, $S^+$, $S^-$.
	\end{itemize}
\end{proof}
\textcolor{red}{\begin{remark}
	If the node set $N$ consists of a single node, then the conditions \eqref{i2:nec1} and \eqref{i2:nec2} of Theorem \ref{thm:facetnec2} reduce to the necessary and sufficient facet conditions of flow pack inequalities given in \cite[Proposition 1]{A:fp}. In this setting, conditions \eqref{i2:nec3}--\eqref{i2:nec6} are no longer relevant. 
\end{remark}}
\begin{theorem}
Let $N=[1,n]$, $E^- = \emptyset$, $d_j > 0$ and $|E_j^+|= 1$, for all $j\in N$ and let the objective set $S^+$ be a path pack for $N$. The necessary conditions in Theorem \ref{thm:facetnec2} along with \begin{enumerate} [(i)]
\item \label{i2:suf1}for each $j\in E^+\setminus S^+$, either $S^+ \cup \{j\}$ is a path cover for $N$ or $\rho_j(S^+) = 0$,
\item \label{i2:suf2}for each $t\in S^+$, there exists $j_t\in E^+\setminus S^+$ such that $S^+\setminus\{t\}\cup \{j_t\}$ is a path cover for $N$,
\item \label{i2:suf5}for each $j \in [1,n-1]$, there exists $k_j \in E^+\setminus S^+$ such that the set $S^+\cup \{k_j\}$ is a path cover and neither node $j$ is backward independent nor node $j+1$ is forward independent for the set $S^+\cup \{k_j\}$
\end{enumerate}
are sufficient for path pack inequality \eqref{ineq:submod2exp} to be facet-defining for $\text{conv}(\mathcal{P})$.

\end{theorem}
\begin{proof}
	We provide $2|E| + n -2$ affinely independent points that lie on the face:
	\begin{align*}
		F = \left\{ (\mathbf{x}, \mathbf{y}, \mathbf{i}, \mathbf{r}) \in \mathcal{P}~:~ 	y(S^+) + \sum_{t\in E^+\setminus S^+} \left(y_t - \min \{ \mu_j, c_t \}x_t \right) = c(S^+)  \right\}.
	\end{align*}
	Let $(\mathbf{y}^*, \mathbf{i}^*, \mathbf{r}^*) \in \mathcal{Q} $ be an optimal solution to (F\ref{opt:vdefn}). Since $S^+$ is a path pack and $E^-=\emptyset$, $v(S^+)= c(S^+)$. Then, notice that $y_t^* = c_t$ for all $t\in S^+$. Moreover, let $e$ be the arc with largest capacity in $S^+$, $\varepsilon > 0$ be a sufficiently small value and $\mathbf{1}_j$ be the unit vector with $1$ at position $j$. First, we give 2$|E^+\setminus S^+|$ affinely independent points represented by $\bar{z}^t = (\bar{\mathbf{x}}^t,\bar{\mathbf{y}}^t,  \bar{\mathbf{i}}^t, \bar{\mathbf{r}}^t)$ and $\tilde{z}^t = (\tilde{\mathbf{x}}^t, \tilde{\mathbf{y}}^t,  \tilde{\mathbf{i}}^t, \tilde{\mathbf{r}}^t)$ for $t\in E^+\setminus S^+$.
	\begin{enumerate} [(i)]
		\item Let $t\in E^+\setminus S^+$, where $S^+\cup \{t\}$ is a path cover for $N$. The solution $\bar{z}^t$ has arcs in $S^+\cup \{t\}$ open, $\bar{x}_j^t = 1$ for $j\in S^+\cup \{t\}$, $0$ otherwise, $\bar{y}_j^t = y_j^*$ for $j\in S^+$ and $\bar{y}_t^t = \rho_t(S^+)$, $0$ otherwise. The forward and backward path arc flow values $\bar{i}_j^t$ and $\bar{r}_j^t$ can then be calculated using flow balance equalities \eqref{const:flowbalance} where at most one of them can be nonzero for each $j\in N$. Sufficiency condition \eqref{i2:nec1} guarantees the feasibility of $\bar{z}^t$.
		\item Let $t\in E^+\setminus S^+$, where $\rho_t(S^+) = 0$ and let $t\neq \ell \in E^+\setminus S^+$, where $S^+\cup \{\ell\}$ is a path cover for $N$. The solution $\bar{z}^t$ has arcs in $S^+\cup \{t,\ell\}$ open, $\bar{x}_j^t = 1$ for $j\in S^+\cup \{t,\ell\}$, and $0$ otherwise, $\bar{y}_j^t = y_j^*$ for $j\in S^+$, $\bar{y}_t^t = 0$, $\bar{y}_{\ell}^t = \rho_{\ell}(S^+)$, and $0$ otherwise. The forward and backward path arc flow values $\bar{i}_j^t$ and $\bar{r}_j^t$ can then be calculated using flow balance equalities \eqref{const:flowbalance} where at most one of them can be nonzero for each $j\in N$. Sufficiency condition \eqref{i2:nec1} guarantees the feasibility of $\bar{z}^t$.

		\item The necessary condition \eqref{i2:nec1} ensures that $\rho_t(S^+) < c_t$, therefore $\bar{y}_t^t < c_t$. In solution $\tilde{z}^t$, starting with $\bar{z}^t$, we send a flow of $\varepsilon$ from arc $t\in E^+\setminus S^+$ to $e\in S^+$. Let $\tilde{\mathbf{y}}^t = \bar{\mathbf{y}}^t + \varepsilon\mathbf{1}_t - \varepsilon\mathbf{1}_e$ and $\tilde{\mathbf{x}}^t = \bar{\mathbf{x}}^t$. If $e < t$, then $\tilde{\mathbf{r}}^t = \bar{\mathbf{r}}^t + \varepsilon \sum_{i=e}^{t-1} \mathbf{1}_i$, $\tilde{\mathbf{i}}^t  = \bar{\mathbf{i}}^t$ and if $e > t$, then $\tilde{\mathbf{i}}^t = \bar{\mathbf{i}}^t + \varepsilon \sum_{i=t}^{e-1} \mathbf{1}_i $, $\tilde{\mathbf{r}}^t  = \bar{\mathbf{r}}^t$.
	\end{enumerate}
	
	Next, we give $2|S^+|-1$ affinely independent feasible points $\hat{z}^t$ and $\check{z}^t$ corresponding to $t\in S^+$ that are on the face $F$. Let $k$ be the arc in $E^+\setminus S^+$ with largest capacity.	
	\begin{enumerate} [(i)] \setcounter{enumi}{3}
	\item In the feasible solutions $\hat{z}^t$ for $e \neq t\in S^+$, we open arcs in $S^+\cup \{k\}$ and send an $\varepsilon$ flow from arc $k$ to arc $t$. Let $\hat{\mathbf{y}}^t = \bar{\mathbf{y}}^k + \varepsilon\mathbf{1}_k - \varepsilon\mathbf{1}_t$ and $\hat{\mathbf{x}}^t = \bar{\mathbf{x}}^k$. If $t < k$, then $\hat{\mathbf{r}}^t = \bar{\mathbf{r}}^k + \varepsilon \sum_{i=t}^{k-1} \mathbf{1}_i$, $\hat{\mathbf{i}}^t  = \bar{\mathbf{i}}^k$ and if $t > k$, then $\hat{\mathbf{i}}^t = \bar{\mathbf{i}}^k + \varepsilon \sum_{i=k}^{t-1} \mathbf{1}_i $, $\hat{\mathbf{r}}^t  = \bar{\mathbf{r}}^k$.
\ignore{	\item Let $\hat{z}^e$ be the solution where the arcs in $S^+$ and $E^+\setminus K^+$ are open and arcs in $E^+\setminus S^+$ are closed. Sufficiency condition \eqref{i2:suf3} ensures that flow balance equalities \eqref{const:flowbalance} can be satisfied without arcs in $E^+\setminus S^+$. As a result, there exist flow values $\hat{y}_t^e$ such that $\sum_{t\in S^+} \hat{y}_t^e = v(S^+)$ and $\sum_{t\in S^+} \hat{y}_t^e = d_{1n}$. Then, $\hat{x}_t^e = 1$ if $t\in S^+\cup (E^+\setminus K^+)$, $0$ otherwise. The forward and backward path arc values are calculated using flow balance equalities where at most one of $\hat{i}_j^e$ and $\hat{r}_j^e$ is non-zero for each $j \in N$.}
	\item In the solutions $\check{z}^t$ for $t\in S^+$, we close arc $t$ and open arc $j_t\in E^+\setminus S^+$ that is introduced in the sufficient condition \eqref{i2:suf2}. Then, $\check{x}_j^t = 1$ if $j\in S^+ \setminus \{t\}$ and if $j=j_t$ and $\check{x}_j^t=0$ otherwise. From sufficient condition \eqref{i2:suf2}, there exists $\check{y}_j^t$ values that satisfy the flow balance equalities \eqref{const:flowbalance}. Moreover, these $\check{y}_j^t$ values satisfy inequality \eqref{ineq:submod2exp} at equality since both $S^+\cup \{j_t\}$ and $S^+\setminus\{t\}\cup \{j_t\}$ are path covers for $N$. Then, the forward and backward path arc flows are found using flow balance equalities where at most one of $\check{i}_j^t$ and $\check{r}_j^t$ are nonzero for each $j\in N$.
	\end{enumerate}
\ignore{
Now, $2|E^+\setminus K^+|-1$ points are provided represented by $\acute{z}^t$ and $\grave{z}^t$ for $t\in E^+\setminus K^+$. Let $a$ be an arbitrary arc in $E^+\setminus S^+$.
\begin{enumerate} [(i)] \setcounter{enumi}{6}
	\item Let $t\in E^+\setminus K^+$. We start with solution $\bar{z}^a$ and open arcs in $E^+\setminus K^+$ one by one. Then, $\acute{\mathbf{x}}^t = \bar{\mathbf{x}}^a + \mathbf{1}_t$, $\acute{\mathbf{y}}^t = \bar{\mathbf{y}}^a$, $\acute{\mathbf{i}}^t = \bar{\mathbf{i}}^a$ and $\acute{\mathbf{r}}^t = \bar{\mathbf{r}}^a$.
	\item In the solution set $\grave{z}^t$, we start with solution $\hat{z}^e$ where the arcs in $S^+\cup (E^+\setminus K^+)$ are open. Let $l\in E^+\setminus K^+$ be an arbitrary arc where $0 < \hat{y}_l^e < c_l$. Existence of such an arc is guaranteed because of sufficiency condition \eqref{i2:suf4}. Then, in solution $\grave{z}^t$, we shift a flow of $\varepsilon$ between arcs $l$ and $t$. For each $t\in E^+\setminus K$,
		\begin{enumerate} [(a)]
			\item if $\hat{y}_t^e = c_t$, then $\grave{\mathbf{y}}^t = \hat{\mathbf{y}}^e + \varepsilon \mathbf{1}_l - \varepsilon \mathbf{1}_t$, $\grave{\mathbf{x}}^t = \hat{\mathbf{x}}^e$, if $l < t$, then $\grave{\mathbf{i}}^t = \hat{\mathbf{i}}^e + \varepsilon + \sum_{i=l}^{t-1} \mathbf{1}_i$ and $\grave{\mathbf{r}}^t = \hat{\mathbf{r}}^e$, if $l > t$, then $\grave{\mathbf{r}}^t = \hat{\mathbf{r}}^e + \varepsilon + \sum_{i=t}^{l-1} \mathbf{1}_i$ and $\grave{\mathbf{i}}^t = \hat{\mathbf{i}}^e$,
			\item if $\hat{y}_t^e < c_t$, then $\grave{\mathbf{y}}^t = \hat{\mathbf{y}}^e - \varepsilon \mathbf{1}_l + \varepsilon \mathbf{1}_t$, $\grave{\mathbf{x}}^t = \hat{\mathbf{x}}^e$, if $l < t$, then $\grave{\mathbf{r}}^t = \hat{\mathbf{r}}^e + \varepsilon + \sum_{i=l}^{t-1} \mathbf{1}_i$ and $\grave{\mathbf{i}}^t = \hat{\mathbf{i}}^e$, if $l > t$, then $\grave{\mathbf{i}}^t = \hat{\mathbf{i}}^e + \varepsilon + \sum_{i=t}^{l-1} \mathbf{1}_i$ and $\grave{\mathbf{r}}^t = \hat{\mathbf{r}}^e$.
		\end{enumerate}
\end{enumerate}}
Finally, we give $n-1$ points $\breve{z}^j$ corresponding to forward and backward path arcs connecting nodes $j$ and $j+1$.
\begin{enumerate}[(i)] \setcounter{enumi}{5}
	\item In the solution set $\breve{z}^j$ for $j=1,\dots,n-1$, starting with solution $\bar{z}^{k_j}$, where $k_j$ is introduced in the sufficient condition \eqref{i2:suf5}, we send a flow of $\varepsilon$ from both forward path arc $(j-1, j)$ and backward path arc $(j,j-1)$. Since the sufficiency condition \eqref{i2:suf5} ensures that $\bar{r}_{j}^{k_j}<b_{j}$ and $\bar{i}_{j}^{k_j}<u_{j}$, the operation is feasible. Let $\breve{\mathbf{y}}^j = \bar{\mathbf{y}}^{k_j}$, $\breve{\mathbf{x}}^j = \bar{\mathbf{x}}^{k_j}$, $\breve{\mathbf{i}}^j = \bar{\mathbf{i}}^{k_j} + \varepsilon \mathbf{1}_j$ and $\breve{\mathbf{r}}^j = \bar{\mathbf{r}}^{k_j} + \varepsilon\mathbf{1}_j$.
\end{enumerate}
\end{proof}
	
\section{Computational study} \label{sec:comp}
We test the effectiveness of path cover and path pack inequalities \eqref{ineq:submod1exp} and \eqref{ineq:submod2exp} by embedding them in a branch-and-cut framework. The experiments are ran on a Linux workstation with 2.93 GHz Intel\textregistered Core\textsuperscript{TM} i7 CPU and 8 GB of RAM  with 1 hour time limit and 1 GB memory limit. The branch-and-cut algorithm is implemented in C++ using IBM's Concert Technology of CPLEX	 (version 12.5). The experiments are ran with one hour limit on elapsed time and 1 GB limit on memory usage. The number of threads is set to one and the dynamic search is disabled. We also turn off heuristics and preprocessing as the purpose is to see the impact of the inequalities by themselves.

\subsection*{Instance generation}
We use a capacitated lot-sizing model with backlogging, where constraints \eqref{const:flowbalance} reduce to: $$i_{j-1} - r_{j-1} + y_j - i_j + r_j = d_j, \quad j \in N.$$ Let $n$ be the total number of time periods and $f$ be the ratio of the fixed cost to the variable cost associated with a non-path arc. The parameter $c$ controls how large the non-path arc capacities are with respect to average demand. All parameters are generated from a discrete uniform distribution. The demand for each node is drawn from the range $[0,30]$ and non-path arc capacities are drawn from the range $[0.75\times c\times \bar{d}, 1.25\times c\times \bar{d}]$, where $\bar{d}$ is the average demand over all time periods. Forward and backward path arc capacities are drawn from $[1.0\times \bar{d},2.0	\times \bar{d}]$ and $[0.3\times \bar{d},0.8\times \bar{d}]$, respectively. The variable costs $p_t$, $h_t$ and $g_t$ are drawn from the ranges $[1,10]$, $[1,10]$ and $[1,20]$ respectively. Finally, fixed costs $f_t$ are set equal to $f\times p_t$. Using these parameters, we generate five random instances for each combination of $n\in \{50, 100, 150\}$, $f\in \{100,200,500,1000\}$ and $c\in \{2, 5, 10\}$.

\subsection*{Finding violated inequalities} Given a feasible solution $(\textbf{x}^*, \textbf{y}^*, \textbf{i}^*, \textbf{r}^*)$ to a linear programming (LP) relaxation of (F\ref{opt:ls}), the separation problem aims to find sets $S^+$ and $L^-$ that maximize the difference
\begin{align*}
		y^*(S^+)- y^*( E^-\setminus L^-) + \sum_{t\in S^+} (c_t - \lambda_j)^+(1-x_t^*) - \sum_{t\in L^-} (\min \{ \lambda_j, c_t \})x_t^* - d_{1n} - c(S^-)
\end{align*}
for path cover inequality \eqref{ineq:submod1exp} and sets $S^+$ and $S^-$ that maximize
\begin{align*}
	y^*(S^+) - y^*(E^-\setminus S^-) - \sum_{t\in E^+\setminus S^+} \min \{c_t, \mu_j \} x_t^* + \sum_{t\in S^-} (c_t-\mu_j)^+(1-x_t)- c(S^+)
\end{align*}
for path pack inequality \eqref{ineq:submod2exp}. We use the knapsack relaxation based heuristic separation strategy described in \cite[pg. 500]{NW88} for flow cover inequalities to choose the objective set $S^+$ with a knapsack capacity $d_{1n}$. Using $S^+$, we obtain the values $\lambda_j$ and $\mu_j$ for each $j\in N$ and let $S^- = \emptyset$ for path cover and path pack inequalities \eqref{ineq:submod1exp} and \eqref{ineq:submod2exp}. For path cover inequalities \eqref{ineq:submod1exp}, we add an arc $t \in E^-$ to $L^-$ if $\lambda_j x_t^* < y_t^*$ and $\lambda_j < c_t$. We repeat the separation process for all subsets $[k,\ell] \subseteq [1,n]$.

\subsection*{Results}
We report multiple performance measures. Let $z_{\text{INIT}}$ be the objective function value of the initial LP relaxation and $z_{\text{ROOT}}$ be the objective function value of the LP relaxation after all the valid inequalities added. Moreover, let $z_{\text{UB}}$ be the objective function value of the best feasible solution found within time/memory limit among all experiments for an instance. Let \texttt{init gap}$=100\times \frac{z_{\text{UB}}-z_{\text{INIT}}}{z_{\text{UB}}}$, \texttt{root gap}$=100\times\frac{z_{\text{UB}}-z_{\text{ROOT}}}{z_{\text{UB}}}$. We compute the improvement of the relaxation due to adding valid inequalities as \texttt{gap imp}$=100\times \frac{\texttt{init gap}-\texttt{root gap}}{\texttt{init gap}}$. We also measure the optimality gap at termination as $\texttt{end gap} = \frac{z_{\text{UB}} - z_{\text{LB}}}{z_{\text{UB}}}$, where $z_{\text{LB}}$ is the value of the best lower bound given by CPLEX. We report the average number of valid inequalities added at the root node under column $\texttt{cuts}$, average elapsed time in seconds under $\texttt{time}$, average number of branch-and-bound nodes explored under $\texttt{nodes}$. If there are instances that are not solved to optimality within the time/memory limit, we report the average end gap and the number of unsolved instances under $\texttt{unslvd}$ next to \texttt{time} results. All numbers except initial gap, end gap and time are rounded to the nearest integers.

In Tables \ref{tab:path}, \ref{tab:cover_vs_pack} and \ref{tab:fc}, we present the performance with the path cover \eqref{ineq:submod1exp} and path pack \eqref{ineq:submod2exp} inequalities under columns \texttt{spi}. To understand how the forward and backward path arc capacities affect the computational performance, we also apply them to the single node relaxations obtained by merging a path into a single node, where the capacities of forward and backward path arcs within a path are considered to be infinite. In this case, the path inequalities reduce to the flow cover and flow pack inequalities. These results are presented under columns \texttt{mspi}.

%\begin{landscape}
% Table generated by Excel2LaTeX from sheet 'Sheet1'
\begin{table}[btbp]
  \centering
  \small
%    \vspace{-0.8cm}
    \caption{Effect of path size on the performance.}
     %of submodular path inequalities applied on paths (\texttt{spi}) and merged paths (\texttt{mspi}), where $p$ represents the path size.}
    \begin{tabular}{cccccccccccccccccc}
    \toprule
          &       &       & &\multicolumn{2}{c}{$p=1$} & \multicolumn{4}{c}{$p \leq 5$} & \multicolumn{4}{c}{$p \leq 0.5\times n$} & \multicolumn{4}{c}{$p \leq n$} \\
    \cmidrule(l){5-6} \cmidrule(l){7-10} \cmidrule(l){11-14} \cmidrule(l){15-18}
    \multicolumn{1}{c}{\multirow{2}[1]{*}{ $n$ }} & \multicolumn{1}{c}{\multirow{2}[1]{*}{ $f$ }} & \multicolumn{1}{c}{\multirow{2}[1]{*}{ $c$ }} & \texttt{init}  & \multicolumn{1}{c}{\texttt{gap imp}} & \multicolumn{1}{c}{\texttt{cuts}} & \multicolumn{2}{c}{\texttt{gap imp}} & \multicolumn{2}{c}{\texttt{cuts}} & \multicolumn{2}{c}{\texttt{gap imp}} & \multicolumn{2}{c}{\texttt{cuts}} & \multicolumn{2}{c}{\texttt{gap imp}} & \multicolumn{2}{c}{\texttt{cuts}} \\
    \cmidrule(l){5-5} \cmidrule(l){6-6} \cmidrule(l){7-8} \cmidrule(l){9-10} \cmidrule(l){11-12} \cmidrule(l){13-14} \cmidrule(l){15-16} \cmidrule(l){17-18}
   \multicolumn{1}{c}{} & \multicolumn{1}{c}{} & \multicolumn{1}{c}{} & \texttt{gap} & \texttt{(m)spi}   & \texttt{(m)spi}   & \texttt{spi}  & \texttt{mspi} & \texttt{spi}  & \texttt{mspi} & \texttt{spi}  & \texttt{mspi} & \texttt{spi}  & \texttt{mspi} & \texttt{spi}  & \texttt{mspi} & \texttt{spi}  & \texttt{mspi} \\
    \midrule
     \multirow{12}[8]{*}{50} & \multirow{3}[2]{*}{100} & 2     & 14.8  & 34\%  & 21    & 87\%  & 52\%  & 212   & 106   & 97\%  & 52\%  & 1164  & 158   & 97\%  & 52\%  & 1233  & 158 \\
          &       & 5     & 44.3  & 56\%  & 52    & 99\%  & 69\%  & 303   & 148   & 99\%  & 69\%  & 664   & 151   & 99\%  & 69\%  & 664   & 151 \\
          &       & 10    & 58.3  & 60\%  & 54    & 96\%  & 70\%  & 277   & 147   & 99\%  & 70\%  & 574   & 167   & 99\%  & 70\%  & 574   & 167 \\
          \cmidrule(l){2-18}
          & \multirow{3}[2]{*}{200} & 2     & 14.5  & 32\%  & 22    & 81\%  & 57\%  & 257   & 133   & 96\%  & 61\%  & 1965  & 241   & 96\%  & 61\%  & 2387  & 241 \\
          &       & 5     & 49.8  & 43\%  & 44    & 99\%  & 57\%  & 378   & 162   & 100\% & 57\%  & 1264  & 171   & 100\% & 57\%  & 1420  & 171 \\
          &       & 10    & 66.3  & 38\%  & 47    & 98\%  & 50\%  & 392   & 169   & 99\%  & 51\%  & 1235  & 197   & 99\%  & 51\%  & 1283  & 197 \\
          \cmidrule(l){2-18}
          & \multirow{3}[2]{*}{500} & 2     & 19.1  & 23\%  & 19    & 77\%  & 48\%  & 266   & 128   & 90\%  & 49\%  & 2286  & 222   & 90\%  & 49\%  & 3249  & 222 \\
          &       & 5     & 54.4  & 35\%  & 36    & 99\%  & 49\%  & 522   & 185   & 100\% & 49\%  & 1981  & 205   & 100\% & 49\%  & 2074  & 205 \\
          &       & 10    & 73.0  & 34\%  & 43    & 99\%  & 40\%  & 498   & 196   & 99\%  & 40\%  & 1336  & 236   & 99\%  & 40\%  & 1445  & 236 \\
          \cmidrule(l){2-18}
          & \multirow{3}[2]{*}{1000} & 2     & 14.6  & 18\%  & 15    & 73\%  & 39\%  & 264   & 99    & 83\%  & 40\%  & 2821  & 211   & 83\%  & 40\%  & 3918  & 212 \\
          &       & 5     & 59.7  & 31\%  & 36    & 98\%  & 45\%  & 487   & 201   & 100\% & 45\%  & 2077  & 239   & 100\% & 45\%  & 2329  & 239 \\
          &       & 10    & 76.9  & 30\%  & 41    & 100\% & 36\%  & 529   & 215   & 100\% & 37\%  & 1935  & 268   & 100\% & 37\%  & 2149  & 268 \\
    \midrule
    \multicolumn{3}{r}{\textbf{Average:}} & \textbf{45.5}	 &	\textbf{36\%} & \textbf{36} &	\textbf{92\%} &	\textbf{51\%} & 	\textbf{365} &	\textbf{157} &	\textbf{97\%} &	\textbf{52\%	} & \textbf{1609} &	\textbf{206}	&\textbf{97\%}	&\textbf{52\%}	& \textbf{1894} &	\textbf{206} \\
    \bottomrule
    \end{tabular}%
  \label{tab:path}%
\end{table}%

In Table \ref{tab:path}, we focus on the impact of path size on the gap improvement of the path cover and path pack inequalities for instances with $n=50$. In the columns under $p=1$, we obtain the same results for both \texttt{mspi} and \texttt{spi} since the paths are singleton nodes.. We present these results under \texttt{(m)spi}. In columns $p\leq q$, we add valid inequalities for paths of size $1,\dots, q$ and observe that as the path size increases, the gap improvement of the path inequalities increase rapidly. On average 97\% of the initial gap is closed as longer paths are used. On the other hand, flow cover and pack inequalities from merged paths reduce about half of the initial gap. These results underline the importance of exploiting path arc capacities on strengthening the formulations. We also observe that the increase in gap improvement diminishes as path size grows. \textcolor{red}{We choose a conservative maximum path size limit of $0.75\times n$ for the experiments reported in Tables \ref{tab:cover_vs_pack}, \ref{tab:fc} and \ref{tab:cpx}.}

\textcolor{red}{In Table \ref{tab:cover_vs_pack}, we investigate the computational performance of path cover and path pack inequalities independently. We present the results for path cover inequalities under columns titled \texttt{cov}, for path pack inequalities under \texttt{pac} and for both of them under the columns titled \texttt{spi}. On average, path cover and path pack inequalities independently close the gap by 63\% and 53\%, respectively. However, when used together, the gap improvement is 96\%, which shows that the two classes of inequalities complement each other very well.}

In Table \ref{tab:fc}, we present other performance measures as well for instances with 50, 100, and 150 nodes. We observe that the forward and backward path arc capacities have a large impact on the performance level of the path cover and pack inequalities. Compared to flow cover and pack inequalities added from merged paths, path cover and path pack inequalities reduce the number of nodes and solution times by orders of magnitude. This is mainly due to better integrality gap improvement (50\% vs 95\% on average).

In Table \ref{tab:cpx}, we examine the incremental effect of path cover and path pack inequalities over the fixed-charge network cuts of CPLEX, namely flow cover, flow path and multi-commodity flow cuts. Under \texttt{cpx}, we present the performance of flow cover, flow path and multi-commodity flow cuts added by CPLEX and under \texttt{cpx\_spi}, we add path cover and path pack inequalities addition to these cuts. We observe that with the addition of path cover and pack inequalities, the gap improvement increases from 86\% to 95\%. The number of branch and bound nodes explored is reduced about 900 times. Moreover, with path cover and path pack inequalities the average elapsed time is reduced to almost half and the total number of unsolved instances reduces from 13 to 6 out of 180 instances.

Tables \ref{tab:path}, \ref{tab:cover_vs_pack}, \ref{tab:fc} and \ref{tab:cpx} show that submodular path inequalities are quite effective in tackling lot-sizing problems with finite arc capacities. When added to the LP relaxation, they improve the optimality gap by 95\% and the number of branch and bound nodes explored decreases by a factor of 1000. In conclusion, our computational experiments indicate that the use of path cover and path pack inequalities is beneficial in improving the performance of the branch-and-cut algorithms.

% Table generated by Excel2LaTeX from sheet 'Sheet2'

\begin{table}[htbp]
  \centering
  \small
  \caption{\textcolor{red}{Effect of path cover (\texttt{cov}) and path pack (\texttt{pac}) inequalities when used separately and together (\texttt{spi}).}}
    {\color{red} \begin{tabular}{cllC{1cm}rrrrrrrrrrrr}
    \toprule
          &       &       &       & \multicolumn{3}{C{2.5cm}}{ \texttt{gap imp}} & \multicolumn{3}{C{2.3cm}}{\texttt{nodes}} & \multicolumn{3}{C{2.3cm}}{\texttt{cuts}} & \multicolumn{3}{C{2cm}}{\texttt{time}} \\
    \cmidrule(l){5-7}     \cmidrule(l){8-10}     \cmidrule(l){11-13}     \cmidrule(l){14-16}   
    \multicolumn{1}{c}{$n$}     & \multicolumn{1}{c}{$f$}     & \multicolumn{1}{c}{$c$}     & \multicolumn{1}{C{1cm}}{\texttt{init gap}} & \multicolumn{1}{c}{\texttt{cov}} & \multicolumn{1}{c}{\texttt{pac}}  & \multicolumn{1}{c}{\texttt{spi}}   & \multicolumn{1}{c}{\texttt{cov}} & \multicolumn{1}{c}{\texttt{pac}}  & \multicolumn{1}{c}{\texttt{spi}}   & \multicolumn{1}{c}{\texttt{cov}} & \multicolumn{1}{c}{\texttt{pac}}  & \multicolumn{1}{c}{\texttt{spi}}   & \multicolumn{1}{c}{\texttt{cov}} & \multicolumn{1}{c}{\texttt{pac}}  & \multicolumn{1}{c}{\texttt{spi}} \\
    \midrule
    \multicolumn{1}{c}{\multirow{12}[0]{*}{50}} & \multicolumn{1}{c}{\multirow{3}[0]{*}{100}} & 2     & 14.8  & 62\%  & 18\%  & 96\%  & 273   & 6258  & 7     & 759   & 13    & 1151  & 0.3   & 0.6   & 0.2 \\
    \multicolumn{1}{c}{} & \multicolumn{1}{c}{} & 5     & 44.3  & 75\%  & 37\%  & 97\%  & 319   & 12366 & 9     & 357   & 25    & 435   & 0.1   & 1.0   & 0.1 \\
    \multicolumn{1}{c}{} & \multicolumn{1}{c}{} & 10    & 58.3  & 77\%  & 39\%  & 93\%  & 213   & 29400 & 63    & 290   & 11    & 386   & 0.1   & 2.1   & 0.1 \\
    \cmidrule(l){2-16} 
    \multicolumn{1}{c}{} & \multicolumn{1}{c}{\multirow{3}[0]{*}{200}} & 2     & 14.5  & 73\%  & 34\%  & 92\%  & 148   & 3268  & 18    & 1593  & 41    & 2469  & 0.7   & 0.3   & 0.6 \\
    \multicolumn{1}{c}{} & \multicolumn{1}{c}{} & 5     & 49.8  & 67\%  & 38\%  & 100\% & 576   & 11525 & 3     & 736   & 36    & 1022  & 0.4   & 0.9   & 0.1 \\
    \multicolumn{1}{c}{} & \multicolumn{1}{c}{} & 10    & 66.3  & 61\%  & 48\%  & 97\%  & 226   & 8799  & 14    & 619   & 32    & 739   & 0.2   & 0.7   & 0.1 \\
    \cmidrule(l){2-16} 
    \multicolumn{1}{c}{} & \multicolumn{1}{c}{\multirow{3}[0]{*}{500}} & 2     & 19.1  & 57\%  & 57\%  & 92\%  & 635   & 1825  & 19    & 1587  & 316   & 2577  & 1.6   & 0.3   & 1.0 \\
    \multicolumn{1}{c}{} & \multicolumn{1}{c}{} & 5     & 54.4  & 56\%  & 75\%  & 99\%  & 348   & 363   & 1     & 902   & 148   & 1164  & 0.3   & 0.1   & 0.1 \\
    \multicolumn{1}{c}{} & \multicolumn{1}{c}{} & 10    & 73.0  & 59\%  & 65\%  & 97\%  & 8410  & 5284  & 11    & 727   & 67    & 698   & 3.0   & 0.5   & 0.1 \\
        \cmidrule(l){2-16} 
    \multicolumn{1}{c}{} & \multicolumn{1}{c}{\multirow{3}[0]{*}{1000}} & 2     & 14.6  & 61\%  & 65\%  & 90\%  & 278   & 258   & 60    & 1362  & 427   & 2094  & 0.8   & 0.3   & 0.9 \\
    \multicolumn{1}{c}{} & \multicolumn{1}{c}{} & 5     & 59.7  & 59\%  & 81\%  & 100\% & 1063  & 208   & 2     & 1673  & 364   & 1792  & 1.3   & 0.1   & 0.1 \\
    \multicolumn{1}{c}{} & \multicolumn{1}{c}{} & 10    & 76.9  & 51\%  & 77\%  & 99\%  & 3791  & 1452  & 5     & 1202  & 155   & 1032  & 2.1   & 0.2   & 0.1 \\
    \midrule
    \multicolumn{3}{c}{\textbf{Average:}} & \textbf{45.5} & \textbf{63\%} & \textbf{53\%} & \textbf{96\%} & \textbf{1357} & \textbf{6751} & \textbf{18} & \textbf{984} & \textbf{136} & \textbf{1297} & \textbf{0.9} & \textbf{0.6} & \textbf{0.3} \\
    \bottomrule
    \end{tabular}}%
  \label{tab:cover_vs_pack}%
\end{table}% 

% Table generated by Excel2LaTeX from sheet 'fc_vs_spi'
\begin{table} %[htbp]
  \centering
  \small
  \caption{Comparison of path inequalities applied to paths (\texttt{spi}) versus applied to merged paths (\texttt{mspi}).}
    %\begin{tabular}{C{0.4cm}L{0.55cm}L{0.4cm}R{0.75cm}R{0.85cm}R{0.7cm}R{0.75cm}R{1.45cm}R{0.9cm}R{0.55cm}R{1cm}L{1cm}R{1cm}L{1cm}}
%    \hspace*{-1cm}
    \begin{tabular}{cllR{0.7cm}rrrrrrrlrl}
    \toprule
          &       &       &       & \multicolumn{2}{c}{ \texttt{gapimp} } & \multicolumn{2}{c}{ \texttt{nodes} } & \multicolumn{2}{c}{ \texttt{cuts} } & \multicolumn{4}{c}{\texttt{time (endgap:unslvd)}} \\
    \cmidrule(l){5-6} \cmidrule(l){7-8} \cmidrule(l){9-10} \cmidrule(l){11-14}
    \multicolumn{1}{c}{$n$ } & \multicolumn{1}{c}{ $f$ } & \multicolumn{1}{c}{ $c$ } &  \texttt{init gap}  & \multicolumn{1}{c}{\texttt{spi}} & \multicolumn{1}{c}{\texttt{mspi}} & \multicolumn{1}{c}{\texttt{spi}} & \multicolumn{1}{c}{\texttt{mspi}} & \multicolumn{1}{c}{\texttt{spi}} & \multicolumn{1}{c}{\texttt{mspi}} & \multicolumn{2}{c}{\texttt{spi}} & \multicolumn{2}{c}{\texttt{mspi}} \\
    \midrule
    \multicolumn{1}{c}{\multirow{12}[8]{*}{50}} & \multicolumn{1}{c}{\multirow{3}{*}{100}} & 2     & 14.8  & 96\%  & 52\%  & 7     & 430   & 1151  & 195   & 0.2   & \multicolumn{1}{l}{} & 0.2   & \multicolumn{1}{l}{} \\
    \multicolumn{1}{c}{} & \multicolumn{1}{c}{} & 5     & 44.3  & 97\%  & 69\%  & 9     & 553   & 435   & 146   & 0.1   & \multicolumn{1}{l}{} & 0.1   & \multicolumn{1}{l}{} \\
    \multicolumn{1}{c}{} & \multicolumn{1}{c}{} & 10    & 58.3  & 93\%  & 70\%  & 63    & 468   & 386   & 160   & 0.1   & \multicolumn{1}{l}{} & 0.1   & \multicolumn{1}{l}{} \\
    \cmidrule(l){2-14}
    \multicolumn{1}{c}{} & \multicolumn{1}{c}{\multirow{3}{*}{200}} & 2     & 14.5  & 92\%  & 59\%  & 18    & 330   & 2469  & 226   & 0.6   & \multicolumn{1}{l}{} & 0.2   & \multicolumn{1}{l}{} \\
    \multicolumn{1}{c}{} & \multicolumn{1}{c}{} & 5     & 49.8  & 100\% & 57\%  & 3     & 1112  & 1022  & 176   & 0.1   & \multicolumn{1}{l}{} & 0.3   & \multicolumn{1}{l}{} \\
    \multicolumn{1}{c}{} & \multicolumn{1}{c}{} & 10    & 66.3  & 97\%  & 53\%  & 14    & 615   & 739   & 173   & 0.1   & \multicolumn{1}{l}{} & 0.2   & \multicolumn{1}{l}{} \\
    \cmidrule(l){2-14}
    \multicolumn{1}{c}{} & \multicolumn{1}{c}{\multirow{3}{*}{500}} & 2     & 19.1  & 92\%  & 43\%  & 19    & 2041  & 2577  & 238   & 1.0   & \multicolumn{1}{l}{} & 0.7   & \multicolumn{1}{l}{} \\
    \multicolumn{1}{c}{} & \multicolumn{1}{c}{} & 5     & 54.4  & 99\%  & 48\%  & 1     & 705   & 1164  & 214   & 0.1   & \multicolumn{1}{l}{} & 0.3   & \multicolumn{1}{l}{} \\
    \multicolumn{1}{c}{} & \multicolumn{1}{c}{} & 10    & 73.0  & 97\%  & 48\%  & 11    & 5659  & 698   & 248   & 0.1   & \multicolumn{1}{l}{} & 1.4   & \multicolumn{1}{l}{} \\
    \cmidrule(l){2-14}
    \multicolumn{1}{c}{} & \multicolumn{1}{c}{\multirow{3}{*}{1000}} & 2     & 14.6  & 90\%  & 45\%  & 60    & 612   & 2094  & 301   & 0.9   & \multicolumn{1}{l}{} & 0.4   & \multicolumn{1}{l}{} \\
    \multicolumn{1}{c}{} & \multicolumn{1}{c}{} & 5     & 59.7  & 100\% & 50\%  & 2     & 2265  & 1792  & 241   & 0.1   & \multicolumn{1}{l}{} & 0.7   & \multicolumn{1}{l}{} \\
    \multicolumn{1}{c}{} & \multicolumn{1}{c}{} & 10    & 76.9  & 99\%  & 40\%  & 5     & 9199  & 1032  & 314   & 0.1   & \multicolumn{1}{l}{} & 2.3   & \multicolumn{1}{l}{} \\
    \midrule
    \multicolumn{1}{c}{\multirow{12}[8]{*}{100}} & \multicolumn{1}{c}{\multirow{3}{*}{100}} & 2     & 13.9  & 95\%  & 65\%  & 39    & 7073  & 3114  & 410   & 1.3   & \multicolumn{1}{l}{} & 3.2   & \multicolumn{1}{l}{} \\
    \multicolumn{1}{c}{} & \multicolumn{1}{c}{} & 5     & 42.2  & 98\%  & 70\%  & 19    & 20897 & 1337  & 297   & 0.2   & \multicolumn{1}{l}{} & 4.8   & \multicolumn{1}{l}{} \\
    \multicolumn{1}{c}{} & \multicolumn{1}{c}{} & 10    & 57.8  & 94\%  & 59\%  & 230   & 395277 & 1298  & 346   & 0.4   & \multicolumn{1}{l}{} & 88.2  & \multicolumn{1}{l}{}\\
        \cmidrule(l){2-14}
    \multicolumn{1}{c}{} & \multicolumn{1}{c}{\multirow{3}{*}{200}} & 2     & 16.1  & 89\%  & 56\%  & 290   & 151860 & 6919  & 478   & 11.0  & \multicolumn{1}{l}{} & 58.4  & \multicolumn{1}{l}{} \\
    \multicolumn{1}{c}{} & \multicolumn{1}{c}{} & 5     & 47.6  & 99\%  & 55\%  & 7     & 455192 & 2355  & 331   & 0.3   & \multicolumn{1}{l}{} & 126.1 & \multicolumn{1}{l}{}\\
    \multicolumn{1}{c}{} & \multicolumn{1}{c}{} & 10    & 65.7  & 95\%  & 54\%  & 104   & 4130780 & 1872  & 399   & 0.5   & \multicolumn{1}{l}{} & 962.3 & (1.1:1) \\
    \cmidrule(l){2-14}
    \multicolumn{1}{c}{} & \multicolumn{1}{c}{\multirow{3}{*}{500}} & 2     & 17.5  & 84\%  & 36\%  & 1047  & 956745 & 11743 & 475   & 47.7  & \multicolumn{1}{l}{} & 390.9 &  \\
    \multicolumn{1}{c}{} & \multicolumn{1}{c}{} & 5     & 53.9  & 99\%  & 41\%  & 4     & 332041 & 3874  & 444   & 0.4   & \multicolumn{1}{l}{} & 115.5 &  \\
    \multicolumn{1}{c}{} & \multicolumn{1}{c}{} & 10    & 72.9  & 96\%  & 42\%  & 34    & 1175647 & 1495  & 474   & 0.3   & \multicolumn{1}{l}{} & 352.5 &  \\
    \cmidrule(l){2-14}
    \multicolumn{1}{c}{} & \multicolumn{1}{c}{\multirow{3}{*}{1000}} & 2     & 17.9  & 91\%  & 41\%  & 173   & 57147 & 10919 & 570   & 21.3  & \multicolumn{1}{l}{} & 23.0  &  \\
    \multicolumn{1}{c}{} & \multicolumn{1}{c}{} & 5     & 58.5  & 100\% & 45\%  & 1     & 284979 & 3261  & 501   & 0.3   & \multicolumn{1}{l}{} & 92.8  &  \\
    \multicolumn{1}{c}{} & \multicolumn{1}{c}{} & 10    & 75.7  & 97\%  & 36\%  & 88    & 3158262 & 2358  & 657   & 0.5   & \multicolumn{1}{l}{} & 1047.0 & (0.7:1) \\
    \midrule
    \multicolumn{1}{c}{\multirow{12}[8]{*}{150}} & \multicolumn{1}{c}{\multirow{3}{*}{100}} & 2     & 13.2  & 94\%  & 64\%  & 336   & 163242 & 5159  & 704   & 11.3  & \multicolumn{1}{l}{} & 107.6 &  \\
    \multicolumn{1}{c}{} & \multicolumn{1}{c}{} & 5     & 44.8  & 99\%  & 65\%  & 17    & 3024118 & 2087  & 431   & 0.5   & \multicolumn{1}{l}{} & 929.6 &  \\
    \multicolumn{1}{c}{} & \multicolumn{1}{c}{} & 10    & 56.9  & 95\%  & 65\%  & 404   & 7254052 & 1492  & 476   & 0.9   & \multicolumn{1}{l}{} & 2087.3 & (0.7:1) \\
    \cmidrule(l){2-14}
    \multicolumn{1}{c}{} & \multicolumn{1}{c}{\multirow{3}{*}{200}} & 2     & 14.7  & 92\%  & 53\%  & 519   & 2772494 & 12636 & 744   & 27.2  & \multicolumn{1}{l}{} & 1390.6 & (0.1:1) \\
    \multicolumn{1}{c}{} & \multicolumn{1}{c}{} & 5     & 48.1  & 99\%  & 55\%  & 15    & 3802938 & 2462  & 508   & 0.6   & \multicolumn{1}{l}{} & 1483.0 & (1.2:2) \\
    \multicolumn{1}{c}{} & \multicolumn{1}{c}{} & 10    & 65.2  & 95\%  & 50\%  & 330   & 9377122 & 2047  & 567   & 0.9   & \multicolumn{1}{l}{} & 3585.9 & (8.2:5) \\
    \cmidrule(l){2-14}
    \multicolumn{1}{c}{} & \multicolumn{1}{c}{\multirow{3}{*}{500}} & 2     & 19.3  & 86\%  & 33\%  & 7927  & 7619674 & 22275 & 792   & 1087.3 & \multicolumn{1}{l}{} & 3165.6 & (4.0:4) \\
    \multicolumn{1}{c}{} & \multicolumn{1}{c}{} & 5     & 54.4  & 100\% & 45\%  & 7     & 7873043 & 4927  & 641   & 0.8   & \multicolumn{1}{l}{} & 2813.6 & (4.3:3) \\
    \multicolumn{1}{c}{} & \multicolumn{1}{c}{} & 10    & 72.3  & 97\%  & 41\%  & 250   & 10219548 & 2678  & 713   & 1.2   & \multicolumn{1}{l}{} & 3422.8 & (11.0:5) \\
    \cmidrule(l){2-14}
    \multicolumn{1}{c}{} & \multicolumn{1}{c}{\multirow{3}{*}{1000}} & 2     & 19.6  & 88\%  & 34\%  & 2824  & 7316675 & 33729 & 724   & 804.8 & \multicolumn{1}{l}{} & 3260.3 & (2.5:3) \\
    \multicolumn{1}{c}{} & \multicolumn{1}{c}{} & 5     & 57.5  & 100\% & 39\%  & 2     & 9661586 & 6710  & 709   & 0.8   & \multicolumn{1}{l}{} & 3578.9 & (9.7:5) \\
    \multicolumn{1}{c}{} & \multicolumn{1}{c}{} & 10    & 75.8  & 96\%  & 37\%  & 99    & 9910056 & 3981  & 829   & 1.2   & \multicolumn{1}{l}{} & 3412.3 & (15.2:5) \\
    \midrule
    \multicolumn{3}{r}{\textbf{Average:}} & \textbf{45.2} & \textbf{95\%} & \textbf{50\%} & \textbf{416} & \textbf{2504012} & \textbf{4619} & \textbf{440} & \textbf{56.3} & \textbf{} & \textbf{903.0} & \textbf{(1.6:36)} \\
	\bottomrule
\end{tabular}
  \label{tab:fc}%
\end{table}%

% Table generated by Excel2LaTeX from sheet 'Sheet3'
\begin{table}[htbp]
  \centering
  \small
  \caption{Effectiveness of the path inequalities when used together with CPLEX's network cuts.}
    \begin{tabular}{cclC{1cm}ccrrrlrl}
    \toprule
    & & & & \multicolumn{2}{c}{\texttt{gapimp}} & \multicolumn{2}{c}{\texttt{nodes}} & \multicolumn{4}{c}{\texttt{time (endgap:unslvd)}} \\
    \cmidrule(l){5-6}    \cmidrule(l){7-8}     \cmidrule(l){9-12}
        $n$  & $f$ & $c$ & \texttt{init gap} & \texttt{cpx\_spi} & \texttt{cpx}   & \multicolumn{1}{c}{\texttt{cpx\_spi}} & \multicolumn{1}{c}{\texttt{cpx}} & \multicolumn{2}{c}{\texttt{cpx\_spi}} & \multicolumn{2}{c}{\texttt{cpx}} \\
    \midrule
    \multicolumn{1}{c}{\multirow{12}[8]{*}{100}} & \multicolumn{1}{c}{\multirow{3}[2]{*}{100}} & 2     & 13.9  & 96\%  & 85\%  & 35    & 1715  & 1.0   &       & 0.5   &  \\
    \multicolumn{1}{c}{} & \multicolumn{1}{c}{} & 5     & 42.2  & 99\%  & 97\%  & 5     & 75    & 0.2   &       & 0.1   &  \\
    \multicolumn{1}{c}{} & \multicolumn{1}{c}{} & 10    & 57.8  & 99\%  & 93\%  & 10    & 2970  & 0.3   &       & 0.6   &  \\
    \cmidrule(l){2-12}
    \multicolumn{1}{c}{} & \multicolumn{1}{c}{\multirow{3}[2]{*}{200}} & 2     & 16.1  & 90\%  & 79\%  & 288   & 9039  & 6.6   &       & 2.1   &  \\
    \multicolumn{1}{c}{} & \multicolumn{1}{c}{} & 5     & 47.6  & 99\%  & 95\%  & 7     & 52    & 0.3   &       & 0.1   &  \\
    \multicolumn{1}{c}{} & \multicolumn{1}{c}{} & 10    & 65.7  & 97\%  & 89\%  & 61    & 3186  & 0.4   &       & 0.7   &  \\
        \cmidrule(l){2-12}
    \multicolumn{1}{c}{} & \multicolumn{1}{c}{\multirow{3}[2]{*}{500}} & 2     & 17.5  & 85\%  & 63\%  & 1232  & 455068 & 57.3  &       & 95.2  &  \\
    \multicolumn{1}{c}{} & \multicolumn{1}{c}{} & 5     & 53.9  & 99\%  & 94\%  & 6     & 92    & 0.4   &       & 0.1   &  \\
    \multicolumn{1}{c}{} & \multicolumn{1}{c}{} & 10    & 72.9  & 98\%  & 89\%  & 11    & 4621  & 0.4   &       & 0.9   &  \\
    \cmidrule(l){2-12}
    \multicolumn{1}{c}{} & \multicolumn{1}{c}{\multirow{3}[2]{*}{1000}} & 2     & 17.9  & 91\%  & 76\%  & 173   & 18109 & 22.2  &       & 3.6   &  \\
    \multicolumn{1}{c}{} & \multicolumn{1}{c}{} & 5     & 58.5  & 100\% & 93\%  & 1     & 156   & 0.3   &       & 0.1   &  \\
    \multicolumn{1}{c}{} & \multicolumn{1}{c}{} & 10    & 75.7  & 97\%  & 85\%  & 117   & 5297  & 0.7   &       & 1.0   &  \\
	\midrule
    \multicolumn{1}{c}{\multirow{12}[8]{*}{150}} & \multicolumn{1}{c}{\multirow{3}[2]{*}{100}} & 2     & 13.2  & 94\%  & 86\%  & 365   & 60956 & 9.7   &       & 19.0  &  \\
    \multicolumn{1}{c}{} & \multicolumn{1}{c}{} & 5     & 44.8  & 100\% & 97\%  & 5     & 119   & 0.4   &       & 0.1   &  \\
    \multicolumn{1}{c}{} & \multicolumn{1}{c}{} & 10    & 56.9  & 99\%  & 92\%  & 16    & 15929 & 0.5   &       & 3.9   &  \\
        \cmidrule(l){2-12}
    \multicolumn{1}{c}{} & \multicolumn{1}{c}{\multirow{3}[2]{*}{200}} & 2     & 14.7  & 92\%  & 80\%  & 954   & 216436 & 44.9  &       & 66.7  &  \\
    \multicolumn{1}{c}{} & \multicolumn{1}{c}{} & 5     & 48.1  & 99\%  & 96\%  & 11    & 284   & 0.5   &       & 0.2   &  \\
    \multicolumn{1}{c}{} & \multicolumn{1}{c}{} & 10    & 65.2  & 97\%  & 91\%  & 181   & 3992  & 0.9   &       & 1.2   &  \\
        \cmidrule(l){2-12}
    \multicolumn{1}{c}{} & \multicolumn{1}{c}{\multirow{3}[2]{*}{500}} & 2     & 19.3  & 86\%  & 69\%  & 7647  & 4943603 & 1049.9 &       & 1215.1 & (0.2:1) \\
    \multicolumn{1}{c}{} & \multicolumn{1}{c}{} & 5     & 54.4  & 100\% & 94\%  & 5     & 5434  & 0.8   &       & 1.6   &  \\
    \multicolumn{1}{c}{} & \multicolumn{1}{c}{} & 10    & 72.3  & 97\%  & 88\%  & 141   & 141211 & 1.4   &       & 35.9  &  \\
        \cmidrule(l){2-12}
    \multicolumn{1}{c}{} & \multicolumn{1}{c}{\multirow{3}[2]{*}{1000}} & 2     & 19.6  & 88\%  & 71\%  & 3051  & 2788993 & 917.4 & (0.2:1) & 619.4 & (0.4:1) \\
    \multicolumn{1}{c}{} & \multicolumn{1}{c}{} & 5     & 57.5  & 100\% & 90\%  & 3     & 4322  & 0.8   &       & 1.2   &  \\
    \multicolumn{1}{c}{} & \multicolumn{1}{c}{} & 10    & 75.8  & 96\%  & 89\%  & 196   & 10588 & 2.5   &       & 2.8   &  \\
    \midrule
    \multicolumn{1}{c}{\multirow{12}[8]{*}{200}} & \multicolumn{1}{c}{\multirow{3}[2]{*}{100}} & 2     & 14.1  & 94\%  & 82\%  & 1623  & 864841 & 32.2  &       & 384.0 &  \\
    \multicolumn{1}{c}{} & \multicolumn{1}{c}{} & 5     & 42.7  & 100\% & 97\%  & 8     & 213   & 0.5   &       & 0.1   &  \\
    \multicolumn{1}{c}{} & \multicolumn{1}{c}{} & 10    & 57.5  & 99\%  & 93\%  & 26    & 45263 & 0.7   &       & 13.8  &  \\
        \cmidrule(l){2-12}
    \multicolumn{1}{c}{} & \multicolumn{1}{c}{\multirow{3}[2]{*}{200}} & 2     & 16.3  & 89\%  & 78\%  & 4279  & 5634851 & 259.9 &       & 1940.4 & (0.1:1) \\
    \multicolumn{1}{c}{} & \multicolumn{1}{c}{} & 5     & 48.0  & 99\%  & 95\%  & 13    & 1310  & 0.9   &       & 0.5   &  \\
    \multicolumn{1}{c}{} & \multicolumn{1}{c}{} & 10    & 65.0  & 98\%  & 90\%  & 128   & 163145 & 1.2   &       & 52.3  &  \\
        \cmidrule(l){2-12}
    \multicolumn{1}{c}{} & \multicolumn{1}{c}{\multirow{3}[2]{*}{500}} & 2     & 16.3  & 88\%  & 72\%  & 8083  & 6805861 & 1226.3 & (0.3:1) & 2137.6 & (0.7:3) \\
    \multicolumn{1}{c}{} & \multicolumn{1}{c}{} & 5     & 54.5  & 99\%  & 93\%  & 7     & 6606  & 1.6   &       & 2.2   &  \\
    \multicolumn{1}{c}{} & \multicolumn{1}{c}{} & 10    & 72.0  & 96\%  & 90\%  & 376   & 900152 & 3.2   &       & 302.4 &  \\
        \cmidrule(l){2-12}
    \multicolumn{1}{c}{} & \multicolumn{1}{c}{\multirow{3}[2]{*}{1000}} & 2     & 18.0  & 82\%  & 63\%  & 13906 & 9894589 & 3000.5 & (1.2:4) & 2835.9 & (3.0:5) \\
    \multicolumn{1}{c}{} & \multicolumn{1}{c}{} & 5     & 57.9  & 100\% & 94\%  & 4     & 1977  & 3.4   &       & 0.8   &  \\
    \multicolumn{1}{c}{} & \multicolumn{1}{c}{} & 10    & 75.6  & 96\%  & 84\%  & 704   & 6127929 & 15.0  &       & 1785.0 & (1.8:2) \\
    \midrule
    \multicolumn{3}{r}{\textbf{Average:}} & \textbf{45.0} & \textbf{95\%} & \textbf{86\%} & \textbf{1213} & \textbf{1087194} & \textbf{185.1} & \textbf{(0.0:6)} & \textbf{320.2} & \textbf{(0.2:13)} \\
    \bottomrule
    \end{tabular}%
 \label{tab:cpx}%
\end{table}% 

\pagebreak

\section*{Acknowledgements}
A. Atamt\"urk and Birce Tezel are supported, in part, the National Science Foundation grant \#0970180 and
by grant FA9550-10-1-0168 from the Office of the Assistant Secretary of Defense for Research and Engineering.
Simge K\"u\c{c}\"ukyavuz is supported, in part,
by the National Science Foundation grant \#1055668. The authors are also thankful to the Associate Editor and anonymous referees for their constructive feedback that improved the paper substantially.

\bibliographystyle{apalike}
\bibliography{reference}
\appendix
\section{Equivalency of (F\ref{opt:vdefn}) to the maximum flow problem} \label{app:equiv}

In Section \ref{sec:spi}, we showed the maximum flow equivalency of $v(S^+,L^-)$ under the assumption that $d_j\geq 0$ for all $j \in N$. In this section, we generalize the equivalency for the paths where $d_j<0$ for some $j \in N$.
\begin{observation} \label{rem:negfl}
If $d_j < 0 $ for some $j\in N$, one can represent the supply amount as a dummy arc incoming to node $j$ (i.e., added to $E_j^+$) with a fixed flow and capacity of $-d_j$ and set the modified demand of node $j$ to be $d_j = 0$.
\end{observation}
Given the node set $N$ with at least one supply node, let $\mathcal{T}(N)$ be the transformed path using Observation \ref{rem:negfl}. Transformation $\mathcal{T}$ ensures that the dummy supply arcs are always open. As a result, they are always in the set $S^+$. We refer to the additional constraints that fix the flow to the supply value on dummy supply arcs as \emph{fixed-flow constraints}. Notice that, $v(S^+, L^-)$ computed for $\mathcal{T}(N)$ does not take fixed-flow constraints into account. In the next proposition, for a path structure, we show that there exists at least one optimal solution to (F\ref{opt:vdefn}) such that the fixed-flow constraints are satisfied.

\begin{proposition} \label{prop:dummysupply}
	Suppose that $d_j <0$ for some $j\in N$. If (F\ref{opt:vdefn}) for the node set $N$ is feasible, then it has at least one optimal solution that satisfies the fixed-flow constraints.
\end{proposition}
\begin{proof}
	We need to show that $v(S^+,L^-)$ has an optimal solution where the flow at the dummy supply arcs is equal to the supply values. The transformation $\mathcal{T}$ makes Proposition \ref{prop:zeros} applicable to the modified path $\mathcal{T}(N)$. Let $\mathcal{Y}$ be the set of optimal solutions of (F\ref{opt:vdefn}). Then, there exists a solution $(\mathbf{y}^*, \mathbf{i}^*, \mathbf{r}^*)\in \mathcal{Y}$ where $y_t^* = 0$ for $t\in E^-\setminus (S^-\cup L^-)$. Let $p\in S_j^+$ represent the index of the dummy supply arc with $c_p = -d_j$. If $y_p^* < c_p$, then satisfying the fixed-flow constraints require pushing flow through the arcs in $E^-\setminus L^-$. We use Algorithm \ref{alg:prep} to construct an optimal solution with $y_p^* = c_p$. Note that each arc in $E_k^-\setminus L_k^-$ for $k\in N$ appear in (F\ref{opt:vdefn}) with the same coefficients, therefore we merge these outgoing arcs into one in Algorithm \ref{alg:prep}. We represent the merged flow and capacity by $\bar{Y}^-_k = \sum_{t\in E_k^-\setminus (S_k^-\cup L_k^-)} y_t^*$ and $\bar{C}_k = c\big(E_k^-\setminus (S_k^-\cup L_k^-)\big)$ for $k\in N$.
	\begin{algorithm}
\caption{} \label{alg:prep}
\begin{algorithmic}
\STATE $\mathcal{J}$: Set of supply nodes in $N$ where the nodes are sorted with respect to their order in $N$.
\STATE $(\mathbf{y}^*, \mathbf{i}^*, \mathbf{r}^*)\in \mathcal{Y}$: $y_t^* = 0$ for all $t \in E^-$.
\FOR{$q \in \mathcal{J}$}
\STATE Let $p$ be the dummy supply arc in $S_q^+$
\STATE $\Delta = c_p - y_p^*$
\FOR{ $j = q$ \TO $n$}
\STATE $ \bar{Y}^-_j = \bar{Y}^-_j + \min\{ \bar{C}_j - \bar{Y}^-_j , \Delta \} $
\STATE $\Delta = \Delta - \min\{ \bar{C}_j - \bar{Y}_j^- , \Delta \}$
\STATE $i_j^{*} = i_j^* + \Delta $
\IF {$i_j^{*} > u_j$}
\STATE $\Delta = i_j^* - u_j$
\STATE $i_j^* = u_j$
\STATE Let $k := j$
\STATE \textbf{break inner loop}
\ENDIF
\ENDFOR
\IF {$\Delta > 0$}
\FOR{ $j = k$ \TO $1$}
\STATE $ \bar{Y}^-_j = \bar{Y}^-_j + \min\{ \bar{C}_j - \bar{Y}^-_j , \Delta \} $
\STATE $\Delta = \Delta - \min\{ \bar{C}_j - \bar{Y}_j^- , \Delta \}$
\STATE $r_j^{*} = r_j^* + \Delta $
\IF {$r_j^{*} > b_j$}
\STATE  $\Delta = r_j^* - b_j$
\STATE \textbf{break inner loop}
\ENDIF
\ENDFOR
\ENDIF
\IF {$\Delta > 0$}
\STATE (F\ref{opt:vdefn}) is infeasible for the node set $N$.
\ENDIF
\ENDFOR
\end{algorithmic}
\end{algorithm}

\end{proof}
Proposition \ref{prop:dummysupply} shows that, under the presence of supply nodes, transformation $\mathcal{T}$ both captures the graph's structure and does not affect (F\ref{opt:ls})'s validity. As a result, Propositions \ref{prop:zeros} and \ref{prop:maxflequiv} become relevant to the transformed path and submodular path inequalities \eqref{ineq:submod1exp} and \eqref{ineq:submod2exp} are also valid for paths where $d_j<0$ for some $j\in N$.

\section{Proofs}
\subsection{Proof of Lemma \ref{prop:separability1}} \label{app:proof_sep1}
	Recall that $C= S^+\cup L^-$ and let $C_1 = S_{N1}^+\cup L_{N1}^-$ and $C_2 = S_{N2}^+\cup L_{N2}^-$. In \eqref{mincutandm}, we showed that the value of the minimum cut is $$v(C) = m_i = \min\{\alpha_i^u+\beta_i^u - c(S_i^+), \alpha_i^d+\beta_i^d-d_i-c(S_i^-)\}$$ for all $i\in N$. For node set $N_1$ and the arc set $C_1$, the value of the minimum cut is $$v_{1}(C_1) = \min \{ \alpha_{j-1}^u+b_{j-1}, \alpha_{j-1}^d \}.$$ This is because of three observations: (1) the values $\alpha_i^{\{u,d\}}$ for $i\in[1,j-2]$ are the same for the node sets $N_1$ and $N$, (2) for the arc set $C_1$ the set $S_{j-1}^+$ now includes the backward path arc $(j,j-1)$ and (3) node $j-1$ is the last node of the first path. Similarly, for node set $N_2$ and the arc set $C_{2}$, the value of the minimum cut is $$v_2(C_2) = \min\{\beta_j^u+u_{j-1}, \beta_j^d\}.$$ For nodes $N_2$ and the arc set $C_2$, (1) the values $\beta_i^{\{u,d\}}$ for $i\in[j+1,n]$ are the same for the node sets $N_2$ and $N$, (2) for the arc set $C_2$ the set $S_{j}^+$ now includes the forward path arc $(j-1,j)$ and (3) node $j$ is the first node of the second path.
	
	Now, if $\alpha_j^u = \alpha_{j-1}^d+u_{j-1}+c(S_j^+)$, then $\alpha_j^d = \alpha_{j-1}^d + d_j + c(S_{j}^-)$ from equations in \eqref{eqn:alphau1}--\eqref{eqn:alphad1}. Then, rewriting $v(C)=m_j$ and $v_1(C_1)$ in terms of $\alpha_{j-1}^d$: $$v(C) = \alpha_{j-1}^d + \min\{\beta_j^u+u_{j-1}, \beta_j^d\}$$ and $$v_1(C_1)=\alpha_{j-1}^d.$$ As a result, the values $v_1(C_1)$ and $v_2(C_2)$ summed gives the value $v(C)$ under the assumption for the value of $\alpha_{j}^u$.
	
	Similarly, if $\beta_{j-1}^u = \beta_{j}^d + b_{j-1} +c(S_{j-1}^+)$, then $\beta_{j-1}^d = \beta_{j}^d +d_{j-1} + c(S_{j-1}^-)$ from equations in \eqref{eqn:betau1}--\eqref{eqn:betad1}. Then, rewriting $v(C)=m_{j-1}$ and $v_2(C_2)$ in terms of $\beta_j^d$: $$v(C) = \beta_j^d+ \min\{\alpha_{j-1}^u +b_{j-1} , \alpha_{j-1}^d\}$$ and $$v_2(C_2)=\beta_j^d.$$ As a result, the values $v_1(C_1)$ and $v_2(C_2)$ summed gives the value $v(C)$ under the assumption for the value of $\beta_{j-1}^u$.

\subsection{Proof of Lemma \ref{prop:separability2}} \label{app:proof_sep2}
The proof follows closely to that of Lemma \ref{prop:separability1}. Let $C= S^+\cup L^-$, $C_1 = S_{N1}^+\cup L_{N1}^-$ and $C_2 = S_{N2}^+\cup L_{N2}^-$. For node set $N_1$ and the arc set $C_1$, the value of the minimum cut is $$v_{1}(C_1) = \min \{ \alpha_{j-1}^u, \alpha_{j-1}^d + u_{j-1} \},$$ where $u_{j-1}$ is added because $c(S_{N1}^-)=c(S_{1j-1}^-)+u_{j-1}$. Similarly, for node set $N_2$ and the arc set $C_{2}$, the value of the minimum cut is $$v_2(C_2) = \min\{\beta_j^u, \beta_j^d + b_{j-1}\},$$ where $b_{j-1}$ is added because $c(S_{N2}^-)=c(S_{jn}^-)+b_{j-1}$.	

Now, if $\alpha_j^d = \alpha_{j-1}^u+b_{j-1}+d_{j-1}+c(S_j^-)$, then $\alpha_j^u = \alpha_{j-1}^u + c(S_{j}^+)$ from equations in \eqref{eqn:alphau1}--\eqref{eqn:alphad1}. Then, rewriting $v(C)=m_j$ and $v_1(C_1)$ in terms of $\alpha_{j-1}^u$: $$v(C) = \alpha_{j-1}^u + \min\{\beta_j^u, \beta_j^d+b_{j-1}\}$$ and $$v_1(C_1)=\alpha_{j-1}^u.$$ As a result, the values $v_1(C_1)$ and $v_2(C_2)$ summed gives the value $v(C)$ under the assumption for the value of $\alpha_{j}^d$.
	
	Similarly, if $\beta_{j-1}^d = \beta_{j}^u + u_{j-1}+ d_{j-1}+c(S_{j-1}^-)$, then $\beta_{j-1}^u = \beta_{j}^u + c(S_{j-1}^+)$ from equations in \eqref{eqn:betau1}--\eqref{eqn:betad1}. Then, rewriting $v(C)=m_{j-1}$ and $v_2(C_2)$ in terms of $\beta_j^u$: $$v(C) = \beta_j^u+ \min\{\alpha_{j-1}^u , \alpha_{j-1}^d+u_{j-1}\}$$ and $$v_2(C_2)=\beta_j^u.$$ As a result, the values $v_1(C_1)$ and $v_2(C_2)$ summed gives the value $v(C)$ under the assumption for the value of $\beta_{j-1}^d$.
	
\subsection{Proof of Lemma \ref{prop:separability3}} \label{app:proof_sep3}
	The proof follows closely to that of Lemmas \ref{prop:separability1} and \ref{prop:separability2}. Let $C= S^+\cup L^-$, $C_1 = S_{N1}^+\cup L_{N1}^-$ and $C_2 = S_{N2}^+\cup L_{N2}^-$. For node set $N_1$ and the arc set $C_1$, the value of the minimum cut is $$v_{1}(C_1) = \min \{ \alpha_{j-1}^u, \alpha_{j-1}^d\}$$ and for node set $N_2$ and the arc set $C_{2}$, the value of the minimum cut is $$v_2(C_2) = \min\{\beta_j^u +u_{j-1}, \beta_j^d + b_{j-1}\}.$$
	
Now, if $\alpha_j^u = \alpha_{j-1}^d+u_{j-1}+c(S_j^+)$ and $\beta_{j-1}^d = \beta_{j}^u + u_{j-1}+ d_{j-1}+c(S_{j-1}^-)$, then $\alpha_j^d = \alpha_{j-1}^d + d_j + c(S_j^-)$ and $\beta_{j-1}^u=\beta_j^u + c(S_j^+)$. Then, rewriting $v(C)=m_j$, $v_1(C_1)$ and $v_2(C_2)$: $$v(C) = \alpha_{j-1}^d + \min\{u_{j-1}+\beta_j^u, \beta_j^d\} = \alpha_{j-1}^d + u_{j-1}+\beta_j^u,$$ $$v_1(C_1)=\alpha_{j-1}^d \text{ and } v_2(C_2)=\beta_j^u+u_{j-1}.$$ As a result, the values $v_1(C_1)$ and $v_2(C_2)$ summed gives the value $v(C)$ under the assumption for the values of $\alpha_{j}^u$ and $\beta_{j-1}^d$.

\subsection{Proof of Lemma \ref{prop:separability4}} \label{app:proof_sep4}
	The proof follows closely to that of Lemmas \ref{prop:separability1} and \ref{prop:separability2}. Let $C= S^+\cup L^-$, $C_1 = S_{N1}^+\cup L_{N1}^-$ and $C_2 = S_{N2}^+\cup L_{N2}^-$. For node set $N_1$ and the arc set $C_1$, the value of the minimum cut is $$v_{1}(C_1) = \min \{ \alpha_{j-1}^u + b_{j-1}, \alpha_{j-1}^d+u_{j-1}\}$$ and for node set $N_2$ and the arc set $C_{2}$, the value of the minimum cut is $$v_2(C_2) = \min\{\beta_j^u, \beta_j^d \}.$$
	
Now, if $\alpha_j^d = \alpha_{j-1}^u+b_{j-1}+d_j + c(S_j^-)$ and $\beta_{j-1}^u = \beta_{j}^d +b_{j-1}+ c(S_{j-1}^+)$, then $\alpha_j^u = \alpha_{j-1}^u + c(S_j^+)$ and $\beta_{j-1}^d=\beta_j^d + d_j + c(S_j^-)$. Then, rewriting $v(C)=m_j$, $v_1(C_1)$ and $v_2(C_2)$: $$v(C) = \alpha_{j-1}^u + \min\{\beta_j^u, \beta_j^d +b_{j-1}\}=\alpha_{j-1}^u + \beta_j^d+b_{j-1},$$ $$v_1(C_1)=\alpha_{j-1}^u + b_{j-1} \text{ and } v_2(C_2)=\beta_j^d.$$ As a result, the values $v_1(C_1)$ and $v_2(C_2)$ summed gives the value $v(C)$ under the assumption for the values of $\alpha_{j}^d$ and $\beta_{j-1}^u$.

\end{document}